\documentclass[a4paper]{article}
\usepackage[utf8]{inputenc}

\title{Covariate-informed reconstruction of partially observed functional data via factor models}

\usepackage{lipsum}
\usepackage{authblk}
\usepackage{tabularx}
\usepackage{amsmath, amsfonts, amssymb, amsthm}
\usepackage[english]{babel}
\usepackage{bm}
\usepackage{bbm}
\usepackage{enumitem}
\usepackage{geometry}
\usepackage{graphicx}
\usepackage{hyperref}
\usepackage{cleveref}
\usepackage{mathtools}
\usepackage{multirow}
\usepackage{nicematrix}
\usepackage{subfig}
\usepackage[bbgreekl]{mathbbol}
\DeclareSymbolFontAlphabet{\bbold}{bbold}
\DeclareSymbolFontAlphabet{\mathbb}{AMSb}%
\newcommand{\bbLambda}{\bbold{\Lambda}}

\usepackage{algorithm}
\usepackage{algpseudocode}

\usepackage{natbib}
\usepackage{multibib}
\newcites{Appendix}{References}

\usepackage{color}

\usepackage{soul}

\DeclarePairedDelimiter\abs{\lvert}{\rvert}
\newcommand{\norm}[1]{\lVert#1\rVert}
\newcommand{\normF}[1]{\lVert#1\rVert_\textup{F}}
\newcommand{\normS}[1]{\lVert#1\rVert_2}

\newcommand{\var}{\textup{Var}}
\newcommand{\expv}{E}
\newcommand{\prob}{P}
\newcommand{\du}{\textup{d}u}
\newcommand{\dv}{\textup{d}v}
\newcommand*\mean[1]{\overline{#1}}
\newcommand{\bigO}{\textsc{O}}
\newcommand{\smallO}{\textsc{o}}
\newcommand{\bigM}{\textsc{M}}

\newcommand{\smallC}{\textsc{c}}

\theoremstyle{remark}

\newtheorem{definition}{Definition}
\theoremstyle{plain}
\newtheorem{theorem}{Theorem}
\newtheorem*{remark}{Remark}
\newtheorem{proposition}{Proposition}
\newtheorem{lemma}{Lemma}
\newtheorem{corollary}{Corollary}

\author{Maximilian Ofner\footnote{\textit{Address for correspondence:} Maximilian Ofner, Institute of Statistics, Graz University of Technology, Kopernikusgasse 24/III, A-8010 Graz, Austria. E-mail: m.ofner@tugraz.at} }
\author{Siegfried Hörmann}
\affil{Graz University of Technology}

\begin{document}

\maketitle

\begin{abstract}
This paper studies linear reconstruction of partially observed functional data which are recorded on a discrete grid. We propose a novel estimation approach based on approximate factor models with increasing rank taking into account potential covariate information. Whereas alternative reconstruction procedures commonly involve some preliminary smoothing, our method separates the signal from noise and reconstructs missing fragments at once. We establish uniform convergence rates of our estimator and introduce a new method for constructing simultaneous prediction bands for the missing trajectories. A simulation study examines the performance of the proposed methods in finite samples. Finally, a real data application of temperature curves demonstrates that our theory provides a simple and effective method to recover missing fragments.
\end{abstract}

\noindent{\it Keywords:} approximate factor models, increasing rank, multivariate functional data, partially observed, uniform consistency.

\section{Introduction}\label{sec:Introduction}

Modern devices' ability to measure data on dense grids leads to a massive amount of high-dimensional data, posing new challenges for researches. In the last decades, growing attention has been put to functional data analysis; see \cite{bosq2000linear}, \cite{ramsay1997functional}, \cite{ferraty2006nonparametric}, and \cite{hsing2015theoretical} for some introductory books. A more recent branch of works deals with partially observed functional data (\cite{kraus2015components}, \cite{kneip2020optimal}, and \cite{delaigle2021estimating} among others), where functions are only observable on subsets of their domain. Such a situation may arise, if an electronic device fails to record data over a certain time span (due to low battery for example). An obvious goal is then to recover the missing data from the available information. In the underlying work, we present a new reconstruction procedure for this purpose and illustrate our methodology with a sample of temperature data.

Figure~\ref{fig:sample} shows intraday temperature curves from Graz (Austria) for $T = 76$ days between July~1 and September 14, 2022 provided by \cite{data2023}. The data were recorded half-hourly at measuring sites in the east and west of the city; we refer to them as Stations~{\tt E} and~{\tt W},  respectively. Whereas curves of Station~{\tt W} are complete, 10 out of 76 days (around 13\%) are incompletely observed at  Station~{\tt E}. For obtaining a complete picture of the data and reducing the risk of loosing important information, it is thus preferable to employ reconstruction procedures to recover the missing parts. With existing methods, the curves of Station~{\tt E} can be recovered using the available subsample of Station {\tt E} only. Here, we propose a new approach which allows to take into account covariate information in our reconstruction problem (in the present example temperature data from Station~{\tt W}). A more detailed discussion of this real data illustration is given in Section~\ref{sec:RealDataApplication}.

\begin{figure}
\centering
\includegraphics[width=\linewidth]{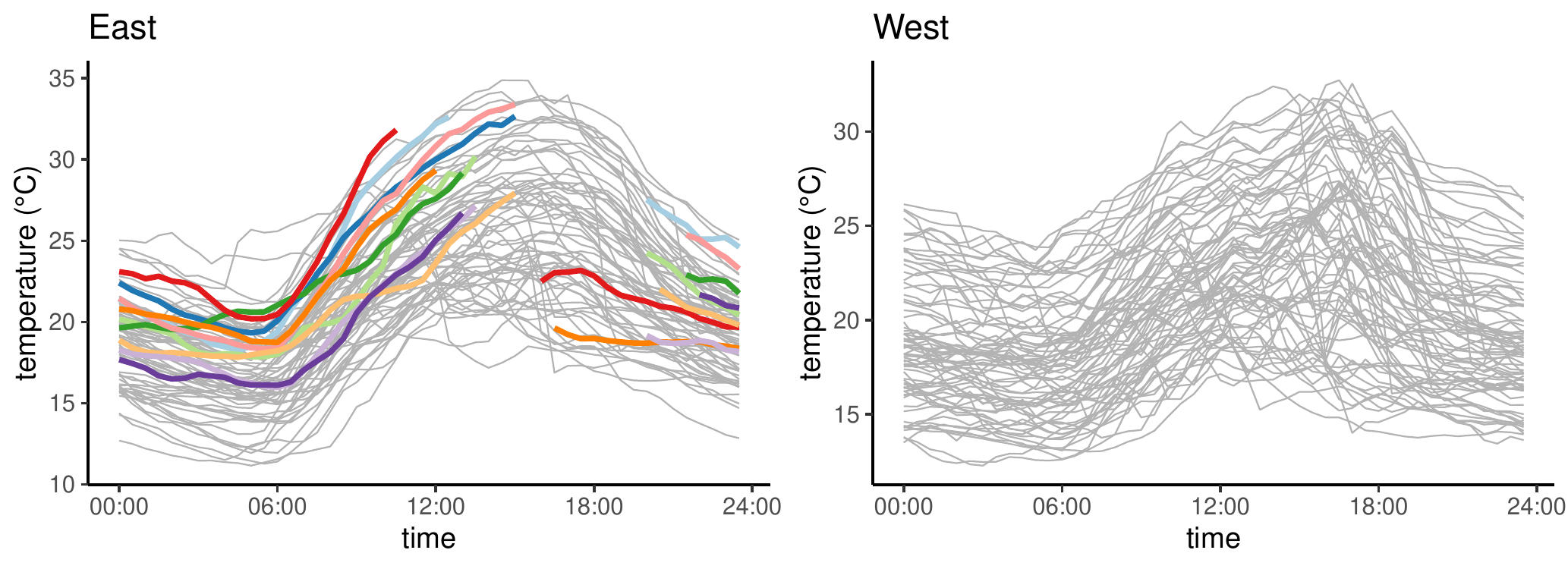}
\caption{Intraday temperature values measured half-hourly in the east and west of Graz~(Austria) between July 1 and September 14, 2022. Highlighted curves reflect 10 days with missing observations.}
\label{fig:sample}
\end{figure}

Generally speaking, we aim to recover an incompletely observed random function $X$ from its observed part and potential covariates $X^{(1)}, \dots, X^{(D)}$ for some~$D \in \mathbb{N}$. Our main interest lies in estimating the best linear reconstruction utilizing a subsample that is observable on the full domain. The main contributions of the underlying work are summarized in the following:

\begin{itemize}
\item First, we allow for additional functional covariates extending related procedures by \cite{kraus2015components} and \cite{kneip2020optimal}. To this end, we make use of a Karhunen-Loève expansion for multivariate functional data. Evidently, the multivariate perspective is needed in settings similar to our real data illustration, where we add covariate information of different measuring sites to improve reconstructions.

\item Second, we formulate direct assumptions on the latent functional structure which implicate an approximate factor model for the corresponding data matrix. This allows us to leverage factor models for the estimation and derivation of sharp convergence rates. Such models are are rather popular in macroeconometrics and successfully applied to many practically relevant examples; see \cite{bai2003inferential}, \cite{fan2013large}, and \cite{bai2021approximate} for important theoretical contributions. From the factor model perspective, our reconstruction task can actually be viewed as a matrix completion problem, which has been recently studied by \cite{bai2021matrix}, \cite{cahan2022factor}, and \cite{xiong2022large}. While in the classical factor models literature the number of factors (denoted by~$r$) is fixed and small, a larger number of factors is theoretically appealing in the case of functional data. In our asymptotic setup, we therefore allow $r$ to diverge with the sample size which poses considerable technical challenges (see \cite{fan2011high}, \cite{li2017determining},  or \cite{hormann2022consistently}).
\end{itemize}

We obtain uniform convergence rates of the proposed reconstruction estimator under mild assumptions. In particular, our theoretical results do not require differentiability of curves and allow for cross-sectionally correlated measurement errors. For functions observable on the full domain, the methodology can be used as a preprocessing tool to separate the signal from the noise which extends results of \cite{hormann2022consistently}. Here, we focus on situations like in our real data example and do not consider the case of functional snippets, where each function is only observable on a proper subset of its domain. This case necessitates extra assumptions as the covariance structure is not identified; see for example \cite{descary2019recovering}, \cite{delaigle2021estimating}, \cite{lin2021basis}, and \cite{lin2022mean}.

The remainder of this work is structured as follows. In Section~\ref{sec:SetupAndNotation}, we discuss multivariate functional data and the data generating process. Thereafter, in Section~\ref{sec:ReconstructionMethodology}, we comment on the reconstruction problem and introduce a factor-based estimation procedure for the best linear reconstruction. Furthermore, we derive asymptotic convergence rates under appropriate assumptions. A method to construct simultaneous prediction bands for the missing parts is then introduced in Section~\ref{sec:PredictionBands}. In Section~\ref{sec:FiniteSampleProps}, we examine finite sample properties in a simulation study and we return to our real data illustration in Section~\ref{sec:RealDataApplication}. Proofs and additional figures can be found in the supplementary file.

\section{Setup and notation}\label{sec:SetupAndNotation}

\subsection{Multivariate functional data}

We study centered continuous random functions $X = (X(u): u \in [0,1])$ (target) and $X^{(d)} = (X^{(d)}(u): u \in [0,1])$ (covariates) indexed by $d \in \{1, \dots, D\}$. For the sake of a simplified presentation, the case of functions defined on different (dimensional) domains as considered in~\cite{happ2018multivariate} is not pursued here. The random elements take values in the separable Hilbert space~$L^2([0,1])$ of square integrable functions. For $f,g \in L^2([0,1])$, the corresponding inner product $\langle f, g\rangle_2 = \int_0^1 f(u)g(u)\,\du$ induces the norm $\norm{f}^2 = \langle f, f\rangle_2 = \int_0^1 f(u)^2\, \du$. We suppose that $X$ were only observable on a subset~$\bigO \subseteq [0,1]$ of its domain (this includes the case~$\bigO = \emptyset$). No information of $X$ is then available on the set $\bigM = [0,1] \setminus \bigO$. Define $X^\smallO = (X(u)\mathbbm{1}\{u \in \bigO\} : u \in [0,1])$ and consider the multivariate random functions
\begin{align*}
&&\mathcal{X}: &\, u \mapsto (X(u), X^{(1)}(u), \dots, X^{(D)}(u))', \\ &&\mathcal{X}^\smallO: & \, u \mapsto (X^\smallO(u), X^{(1)}(u), \dots, X^{(D)}(u))'.
\end{align*}
We may view $\mathcal{X}$ and $\mathcal{X}^\smallO$ as random elements which map to~$\mathcal{H} = \prod_{d=0}^D L^2([0,1])$. Using positive weights $w_d > 0$, one can show that $\mathcal{H}$ equipped with the inner product
\begin{equation*}
\langle \langle f, g\rangle \rangle = \sum_{d=0}^D w_d \langle f^{(d)}, g^{(d)}\rangle_2, \qquad f, g \in \mathcal{H},
\end{equation*}
is a Hilbert space. Here, we set $w \equiv 1$ and note that one could replace $X^{(d)}$ with~$\sqrt{w_d}X^{(d)}$ for the general case. An empirical choice of the weights is discussed in Section~\ref{sec:ChoiceOfParameters}. We expand $\mathcal{X}^\smallO$ in terms of its multivariate Karhunen-Loève expansion
\begin{equation}\label{eq:MKL}
\mathcal{X}^\smallO(u) = \sum_{k=1}^\infty \xi_k^\smallO \varphi_k^\smallO(u), \qquad u \in [0,1],
\end{equation}
where the scores $\xi^\smallO_k = \langle \langle \mathcal{X}^\smallO, \varphi^\smallO_k\rangle \rangle$ are uncorrelated random variables with variance~$\lambda^\smallO_k$, and~$(\lambda_k^\smallO, \varphi_k^\smallO)$ constitute eigenelements of the corresponding covariance operator $\Gamma^\smallO: \mathcal{H} \to \mathcal{H}$ of $\mathcal{X}^\smallO$ given by
\begin{equation*}
\Gamma^\smallO (f) = \expv[ \langle \langle \mathcal{X}^\smallO, f \rangle \rangle \mathcal{X}^\smallO], \qquad f \in \mathcal{H}.
\end{equation*}
Throughout the underlying work, we assume that the sequence of eigenvalues is ordered, that is to say, $\lambda^\smallO_1\geq \lambda^\smallO_2 \geq \dots \geq 0$. Note that eigenvalues, eigenfunctions and scores depend on the observation set $\bigO$. The Karhunen-Loève expansion is optimal in the sense that it minimizes the mean squared error which results from truncating~\eqref{eq:MKL} at a finite level. The truncation error
\begin{equation}\label{eq:TruncationError}
\mathcal{R}_r^\smallO(u) = \mathcal{X}^\smallO(u) - \sum_{k=1}^r \xi_k^\smallO \varphi_k^\smallO(u), \qquad u \in [0,1],
\end{equation}
satisfies $\mathcal{R}_r^\smallO(u) \to 0$ as $r \to \infty$, where convergence is in mean square and uniformly for $u \in [0,1]$. General properties of the multivariate Karhunen-Loève expansion are discussed in \cite{happ2018multivariate}. 

\subsection{Data generating process}

In the following, we describe the data generating process in more detail. We generally allow for a time dependent sample of functional data as stated in Assumption~\ref{ass:Sample}.
\begin{enumerate}[label = (A\arabic*)]
\item \label{ass:Sample} $(\mathcal{X}_t: t\leq T)$ is a strictly stationary $\alpha$-mixing series of multivariate functions.
\end{enumerate}
For  convenience, the definition of $\alpha$-mixing is given in Section~\ref{sec:Mixing} of the supplementary file.

\subsubsection{Discrete measurements}

Adopting a common framework, we suppose that functions are recorded on a discrete grid with additive noise. Accordingly, for~$d \in \{0, \dots, D\}$, $t \in \{1, \dots, T\}$, and $i \in \{1, \dots, N\}$ we set
\begin{equation}\label{eq:Measurements}
y^{(d)}_{ti} = x^{(d)}_{ti} + e^{(d)}_{ti},
\end{equation}
where $x^{(0)}_{ti} = X_t(u_i)$, $x^{(d)}_{ti} = X^{(d)}_t(u_i)$ for $d \geq 1$, and $0 = u_1 < u_2 < \dots < u_N  = 1$ defines a grid of equispaced points. In addition, let
\begin{equation*}
\mathbbm{y}_t = (y_{t1}^{(0)}, \dots, y_{tN}^{(0)}, \dots,y_{t1}^{(D)}, \dots, y_{tN}^{(D)})
\end{equation*}
denote the vector of size $(1 \times (D+1)N)$ which is obtained by stacking the measurements in~\eqref{eq:Measurements}. In our asymptotic setup, we assume that the number of grid points diverges, that is, $N \to \infty$. If $N$ remains bounded, there is no hope of recovering $X$ even in the absence of noise, as there exist infinitely many random functions which marginal distributions coincide on $u_1, \dots, u_N$; see \cite{hall2006properties}.

\begin{remark}
The assumption of equispaced grid points is often satisfied in practice, when measurements are taken by electronic instruments; see our real data application. However, our theoretical results still remain true if $X, X^{(1)}, \dots, X^{(D)}$ were recorded on $D + 1$ different but sufficiently dense grids. We do not cover the case of sparse measurements but deliberately focus on the dense observation regime where we can effectively employ factor models.
\end{remark}

\subsubsection{Missing data}

Now suppose that a curve $X_s$ were only \emph{observable} on a set~$\bigO = \bigO_s \subset [0,1]$. By~\eqref{eq:Measurements}, it holds that $y_{si}^{(0)} = X_s(u_i) + e^{(0)}_{si}$ is observed whenever $u_i \in \bigO$ and thus $\mathbbm{y}_s$ may be prone to missing values in the leading entries.
Concerning the observation sets, we assume

\begin{enumerate}[label = (A\arabic*)]
\setcounter{enumi}{1}
\item \label{ass:MCAR} $(\bigO_t: t\leq T)$ are i.i.d.~and independent from all other quantities.
\end{enumerate}
Assumption~\ref{ass:MCAR} is a missing-completely-at-random condition (\cite{rubin1976inference}) and similarly imposed in the related works of \cite{kraus2015components} and \cite{kneip2020optimal}. A relaxation of the assumption is discussed in \cite{liebl2019partially}. 

To recover the missing part, we utilize a completely observable subsequence of $(\mathbbm{y}_t: t \leq T)$. Let $\mathcal{T} = \{t \leq T: \bigO_t = [0,1]\}$ be the corresponding indices of completely observable curves and set $T_\smallC = \vert \mathcal{T}\vert$. In addition to $\mathbbm{y}_t$, let~$\mathbbm{y}_{\smallO,t}$ be the vector of dimension $(1\times N_\smallO)$ which stacks $(y^{(0)}_{ti}: u_i \in \bigO)$ with $(y_{t1}^{(1)}, \dots, y_{tN}^{(1)}, \dots,y_{t1}^{(D)}, \dots, y_{tN}^{(D)})$. Set $\mathbbm{Y}_\smallC$ and~$\mathbbm{Y}_\smallO$ to be matrices with rows~$\mathbbm{y}_t$ and $\mathbbm{y}_{\smallO,t}$ for~$t \in \mathcal{T}$, respectively. Let $\mathbbm{y}_{\smallC,i}$ be the $i$-th column of $\mathbbm{Y}_\smallC$. Define $\mathbbm{X}_\smallO$ and $\mathbbm{E}_\smallO$ and other quantities in the same vein. An illustration of the notation is given in Figure~\ref{fig:Illustration}.

\begin{figure}
\noindent\begin{tabularx}{\textwidth}{@{}c@{}cr} 
$X$ & $X^{(1)}$ & \\
\raisebox{-20pt}{\includegraphics[width = 0.46\textwidth]{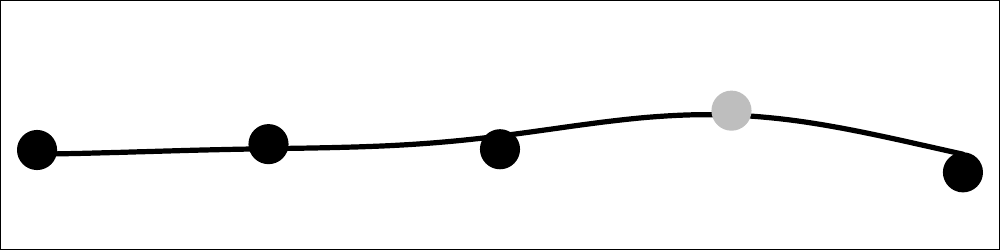}} & \raisebox{-20pt}{\includegraphics[width = 0.46\textwidth]{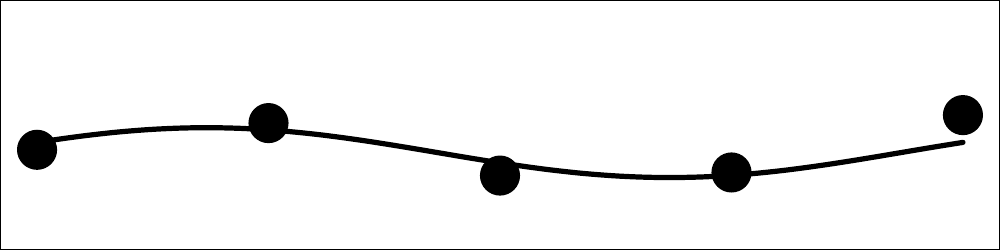}} & $\mathbbm{y}_1$\\
\raisebox{-20pt}{\includegraphics[width = 0.46\textwidth]{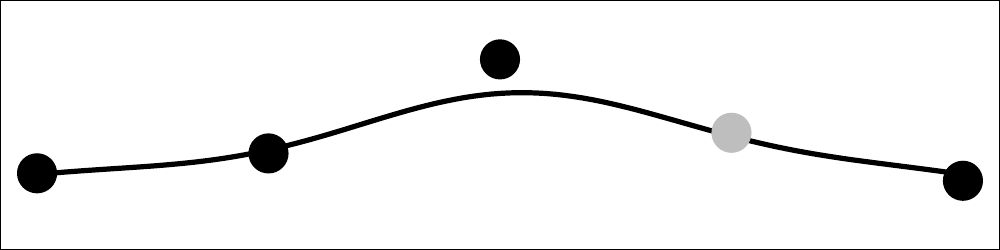}} & \raisebox{-20pt}{\includegraphics[width = 0.46\textwidth]{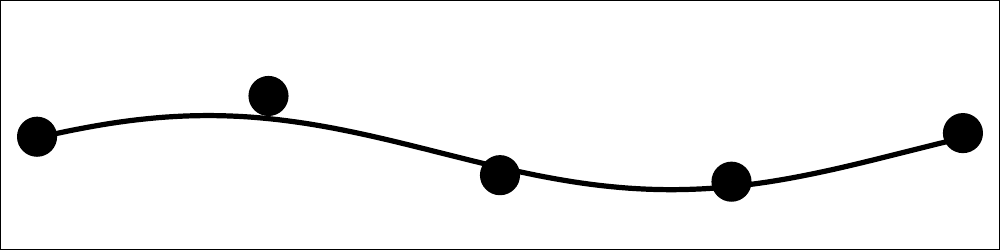}} & $\mathbbm{y}_2$ \\
\raisebox{-40pt}{\includegraphics[width = 0.46\textwidth]{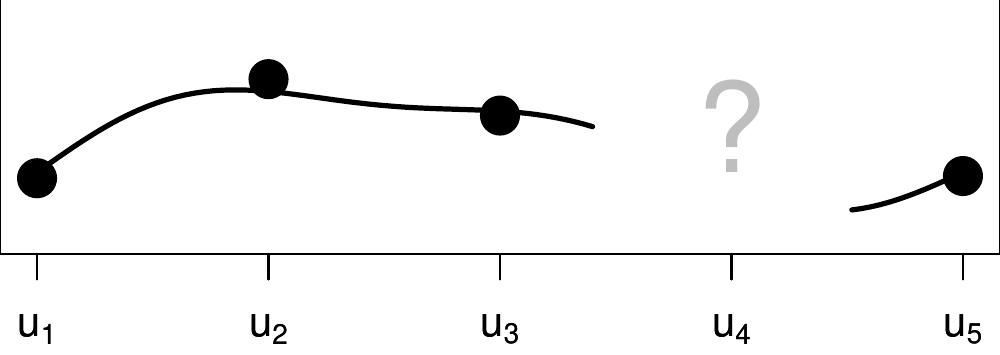}} & \raisebox{-40pt}{\includegraphics[width = 0.46\textwidth]{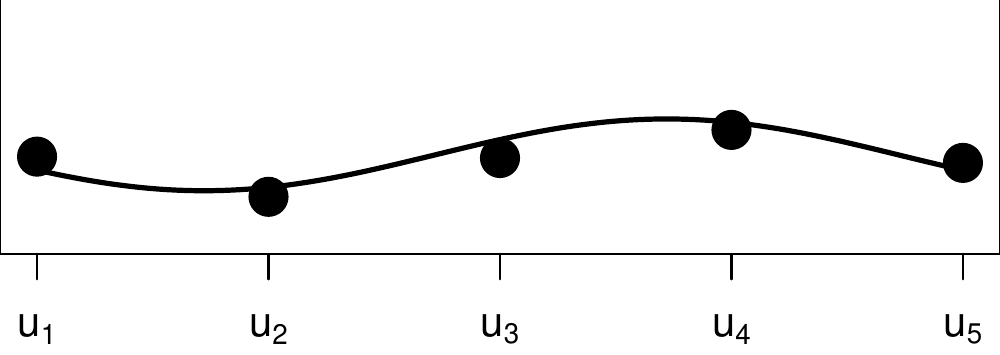}} & $\mathbbm{y}_3$ \\
\includegraphics[width = 0.46\textwidth]{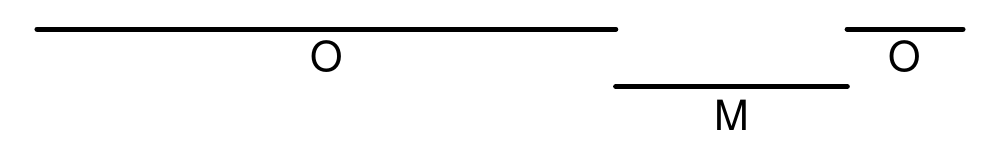} &  &
\end{tabularx}
\caption{The figure considers a sample of $T = 3$ multivariate functions with $D = 1$ covariate. Each curve is measured on a regular grid of $N = 5$ points with noise. Stacking the measurements row-wise gives $(1 \times 10)$ vectors $\mathbbm{y}_t$. As only the first two curves are completely observable, the index set of complete curves equals $\mathcal{T} = \{1,2\}$ and $\mathbbm{Y}_\smallC = (\mathbbm{y}_1, \mathbbm{y}_2)'$. The fourth element in $\mathbbm{y}_3$ is missing since~$X_3$ is not observable on $\bigM \ni u_4$. The  vectors $\mathbbm{y}_{\smallO,t}$ are obtained by stacking black bullets row-wise and of size $(1 \times 9)$. We get $\mathbbm{Y}_\smallO = (\mathbbm{y}_{\smallO, 1}, \mathbbm{y}_{\smallO, 2})'$.}
\label{fig:Illustration}
\end{figure}
 
\subsubsection{Factor model}

When $N$ is large, a dimension reduction step is needed to perform statistical inference. Let us gain some more intuition before we specify our model in Assumption~\ref{ass:FFM} stated below. In view of~\eqref{eq:TruncationError} and optimality of the Karhunen-Loève expansion, a finite number of scores $\xi_k^\smallO$ is often sufficient to capture the main variation of the random functions reasonably well. If $\mathcal{R}_r^\smallO$ in~\eqref{eq:TruncationError} is zero for some~$r$, we may view $\mathcal{X}_t^\smallO(u) = \sum_{k=1}^r \xi_k^\smallO \varphi_k^\smallO(u)$ as the ``common component'' in a functional factor model. This, in turn, implies that the related data matrix $\mathbbm{X}_\smallO$ has rank $r$. Under suitable assumptions on the measurement errors, $\mathbbm{Y}_\smallO = \mathbbm{X}_\smallO + \mathbbm{E}_\smallO$ then follows a multivariate factor model. Concerning the ``idiosyncratic component'' $\mathbbm{E}_\smallO$, one typically assumes independent or mildly correlated elements (in the later case, the factor model is termed \textit{approximate}). To leverage factor analysis in our work, we thus need to control the rank of $\mathbb{X}_\smallO$. Motivated by the above discussion, we require a functional factor model in Assumption~\ref{ass:FFM} to ensure $\text{rank}(\mathbb{X}_\smallO) = r$.

\begin{enumerate}[label = (A\arabic*)]
\setcounter{enumi}{2}
\item\label{ass:FFM}
There exists some $r$ such that the truncation error in~\eqref{eq:TruncationError} satisfies $\mathcal{R}_r^\smallO \equiv 0$.
\end{enumerate}

Assumption~\ref{ass:FFM} is only a sufficient condition for $\mathbbm{Y}_\smallO$ to follow a factor model. In principle, the assumption $\text{rank}(\mathbb{X}_\smallO) = r$ would suffice for obtaining our theoretical results. Since $\mathcal{R}_r^\smallO$ typically vanishes rapidly as $r \to \infty$, Assumption~\ref{ass:FFM} should be approximately satisfied in most practical applications. To also accomodate the functional nature, we allow $r$ to diverge in our asymptotic scenario and assume that the sequence of eigenvalues is summable,~$\lim_{r\to \infty} \sum_{k=1}^r \lambda_k^\smallO < \infty$.

By the above, we can rewrite $\mathbbm{Y}_\smallO$ as
\begin{equation}\label{eq:FactorModel}
\mathbbm{Y}_\smallO = \mathbbm{X}_\smallO + \mathbbm{E}_\smallO = \mathbbm{F}_\smallO \bbLambda_\smallO' + \mathbbm{E}_\smallO,
\end{equation}
where $\mathbbm{F}_\smallO$ is a $(T_\smallC \times {r})$ matrix with elements $\xi^\smallO_{tk}/\sqrt{\lambda^\smallO_k}$ and the columns of the~$(N_\smallO \times r)$ matrix~$\bbLambda_\smallO$ stack the elements $\sqrt{\lambda^\smallO_k}{\varphi_k^\smallO}^{(0)}(u_i)$, for $u_i \in \bigO$, with the vector $$\sqrt{\lambda^\smallO_k}\left({\varphi^\smallO_k}^{(1)}(u_{1}), \dots,{\varphi^\smallO_k}^{(1)}(u_{N}), \dots, {\varphi^\smallO_k}^{(D)}(u_{1}), \dots,{\varphi^\smallO_k}^{(D)}(u_{N}) \right).$$
If $\mathbbm{E}_\smallO$ is independent of $\mathbbm{X}_\smallO$ and if the largest eigenvalue of $\expv[\mathbbm{E}_\smallO'{\mathbbm{E}_\smallO}]$ grows at most linear in~$T_\smallC$, then \eqref{eq:FactorModel} constitutes an approximate factor model of rank~$r$ (\cite{chamberlain}). Main contributions to the asymptotic theory in the case of large dimensions include \cite{bai2003inferential} and \cite{fan2013large}. While the number of covariates $D$ is considered to be fixed in our work, we ensure asymptotic identifiability by letting $N, T \to \infty$.

\begin{remark}
The literature usually assumes a strong factor structure which means that~$\bbLambda_\smallO'\bbLambda_\smallO/N$ converges to a positive definite matrix. In contrast, our assumptions with summability of eigenvalues entail that $\lambda_{\textup{min}}(\bbLambda_\smallO'\bbLambda_\smallO/N) \approx \lambda_r^\smallO$ gets arbitrarily small as $r \to \infty$. This should be compared to factor models with weak loadings where $\bbLambda_\smallO'\bbLambda_\smallO/N^\alpha$ only has a positive definite limit for some~$\alpha \in (0,1)$. As shown by \cite{bai2021approximate}, estimators are still consistent in the later case.
\end{remark}

\section{Reconstruction methodology}\label{sec:ReconstructionMethodology}

As a next step, we propose a linear reconstruction procedure for the missing fragments in Section~\ref{sec:LinearReconstructionOperator}. An estimation procedure and corresponding convergence rates are then discussed in Section~\ref{sec:Estimation} and~\ref{sec:ConvergenceRates}, respectively.

\subsection{Linear reconstruction operator} \label{sec:LinearReconstructionOperator}

To reconstruct $X$, we seek a linear operator $\mathcal{\ell}: \mathcal{H} \to L^2([0,1])$ which minimizes
\begin{equation}\label{eq:MinimizationProblem}
\expv[(X(u) - \mathcal{\ell}(\mathcal{X}^\smallO)(u))^2]
\end{equation}
at any $u \in [0,1]$. Invoking a linear regression, the solution to the above minimization problem is given by the finite-rank operator
\begin{equation}\label{eq:Operator}
\mathcal{L}(\mathcal{X}^\smallO)(u) = \sum_{k=1}^r \xi^\smallO_k \tilde{\varphi}^\smallO_k(u), \qquad u \in [0,1],
\end{equation}
where the ``extrapolated'' basis functions $\tilde{\varphi}^\smallO_k: [0,1] \to \mathbb{R}$ are defined as
\begin{equation*}
\tilde{\varphi}^\smallO_k(u) = \frac{\expv[X(u)\xi^\smallO_k]}{\lambda^\smallO_k}, \qquad u \in [0,1].
\end{equation*}
Continuity of $X(u)$ entails continuity of $\mathcal{L}(\mathcal{X}^\smallO)(u)$. Moreover, for $u \in \bigO$, it can easily be seen that $\mathcal{L}(\mathcal{X}^\smallO)(u) = X(u)$. Properties of linear reconstruction operators in the univariate setting are discussed in \cite{kneip2020optimal}. Due to Assumption~\ref{ass:FFM}, the reconstruction operator given in~\eqref{eq:Operator} is bounded.

The reconstruction of $X$ clearly relates to extrapolation which is substantially different from the interpolation problem thoroughly studied in the functional data literature (see for instance \cite{yao2005functional}). Note that $\mathcal{L}(\mathcal{X}^\smallO)$ is only a linear approximation to $X$ and one cannot hope to reconstruct $X$ from $\mathcal{X}^\smallO$ without error as $X$ is not even identifiable from $\mathcal{X}^\smallO$ in general. Rather, it will be the case that
\begin{equation}\label{eq:ReconstructionError}
X(u) = \mathcal{L}(\mathcal{X}^\smallO)(u) + Z(u), \qquad u \in [0,1],
\end{equation}
where $Z = (Z(u): u \in [0,1])$ denotes the ``reconstruction error" which vanishes on the set~$\bigO$. The following proposition can be proven along the lines of Theorem~2.3 in \cite{kneip2020optimal}.

\begin{proposition}\label{prop}
It holds that
\begin{enumerate}[label=(\roman*)]
\item $\expv[\mathcal{X}^\smallO(u) Z(v)] \equiv 0$ for any $u,v \in [0,1]$;
\item the operator $\mathcal{L}$ minimizes \eqref{eq:MinimizationProblem} among all linear operators $\ell: \mathcal{H} \to L^2([0,1])$. 
\end{enumerate}
\end{proposition}
Proposition~\ref{prop} states that (i) $\mathcal{X}^\smallO$ is uncorrelated with the reconstruction error and (ii) $\mathcal{L}$ is the optimal linear reconstruction operator with respect to the mean squared loss.

\subsection{Estimation}\label{sec:Estimation}

The intrinsic factor model opens a convenient path for estimating the best linear reconstructions. We follow a general approach and estimate the factors $\mathbbm{F}_\smallO$ via principal components. To this end, consider the singular value decomposition
\begin{equation*}
\frac{1}{\sqrt{N T_\smallC}} \mathbbm{Y}_\smallO \approx \mathbbm{U}_\smallO \mathbbm{D}_\smallO \mathbbm{V}_\smallO',
\end{equation*}
where $\mathbbm{D}_\smallO$ is the diagonal matrix containing the first ${r}$ singular values of $\frac{1}{\sqrt{N T_\smallC}} \mathbbm{Y}_\smallO$ in descending order. The columns of $\mathbbm{U}_\smallO$ and $\mathbbm{V}_\smallO$ contain the first ${r}$ left and right eigenvectors, respectively. We then set
\begin{equation}\label{eq:MatrixEstimators}
\widehat{\mathbbm{F}}_\smallO = \sqrt{T_\smallC} \mathbbm{U}_\smallO, \quad \widehat{\mathbbm{f}}_{\smallO,s} = \frac{1}{\sqrt{N}} \mathbbm{y}_{\smallO,s} \mathbbm{V}_\smallO \mathbbm{D}_\smallO^{-1}, \quad \widehat{\mathcal{L}(\mathcal{X}^\smallO_s)}(u_i) = \widehat{\mathbbm{f}}_{\smallO,s} (\widehat{\mathbbm{F}}_\smallO'\widehat{\mathbbm{F}}_\smallO)^{-1} \widehat{\mathbbm{F}}_\smallO' \mathbbm{y}_{\smallC,i},
\end{equation}
adopting the convention that $\frac{1}{T_\smallC}{\widehat{\mathbbm{F}}_\smallO}' {\widehat{\mathbbm{F}}}_\smallO = \mathbbm{I}_{r}$.
\begin{remark}\phantom{}
\begin{itemize}
\item Pursuing the illustration given in Figure~\ref{fig:Illustration}, the missing observation at the interrogation mark ``\textcolor{lightgray}{\textbf{?}}'' is thus imputed by an estimate for $\mathcal{L}(\mathcal{X}^\smallO_3)(u_4)$ which results from a regression of the gray bullets ``\textcolor{lightgray}{$\bullet$}" onto the factor scores of the matrix $\mathbbm{Y}_\smallO$.
\item The procedure relates to a factor-based imputation procedure for multivariate data introduced in \cite{cahan2022factor}; see also \cite{bai2021matrix}. However, the authors essentially  assume in our notation that~$\mathbbm{X}_\smallC$ and~$\mathbbm{X}_\smallO$ have equal rank~$r$. By careful analysis of the reconstruction error, we generalize the concept to the practically relevant case of $\textup{rank}(\mathbbm{X}_\smallC) \gg \textup{rank}(\mathbbm{X}_\smallO) = r \to \infty$, which appears to be a result of its own interest.
\item A separate factor model is fitted for each partially observed curve. While this adds computational cost, it is essential for proper reconstructions as the optimal linear reconstruction depends on the observation set.
\end{itemize}

\end{remark}
An estimator for the continuous reconstruction is then obtained by linear interpolation,
\begin{equation}\label{eq:Estimator}
\begin{aligned}
\widehat{\mathcal{L}(\mathcal{X}^\smallO_s)}(u) &= \sum_{i = 1}^{N-1} \left(\widehat{\mathcal{L}(\mathcal{X}^\smallO_s)}(u_i) + \frac{u - u_i}{u_{i+1}-u_i}(\widehat{\mathcal{L}(\mathcal{X}^\smallO_s)}(u_{i+1})-\widehat{\mathcal{L}(\mathcal{X}^\smallO_s)}(u_i))   \right) \\
&\qquad \times \mathbbm{1} \{u \in [u_i,u_{i+1})\}.
\end{aligned}
\end{equation}
Our estimators depend on a rank $r$ which is typically unknown and must be estimated from the data. The literature on approximate factor models considers several approaches; see for instance \cite{bai2002determining} or \cite{onatski2010determining}. Moreover, \cite{li2017determining} prove consistency of a criterion similar to \cite{bai2002determining} in the case of growing $r \to \infty$. An adaption of this result to our setting falls outside the scope of the underlying work and is of future interest. In what follows, we assume that we were given an estimator $\widehat{r}$ such that $\prob(\widehat{r} = r) \to 1$. In Section~\ref{sec:ChoiceOfParameters}, we choose the number of factors via 5-fold cross-validation.

\subsection{Convergence rates}\label{sec:ConvergenceRates}

For clearer presentation, we suppose that~$T_\smallC \asymp T$ ($T_\smallC/T$ is bounded from above and below) and define $X^{(0)}(u) = X(u)$. Let $C > 1$ be some large absolute constant which may have a different value at each appearance. Consider the following assumptions.

\begin{enumerate}[label = (A\arabic*)]
\setcounter{enumi}{3}
\item \label{ass:Mixing} The mixing coefficients of $(\mathcal{X}_t: t \leq T)$ satisfy $\alpha(h) = \exp(-Ch^{\gamma_1})$ for some $\gamma_1 > 0$;
\item \label{ass:MomentsX} there is some $\delta>0$ such that $\sup_{u} \expv[X_t^{(d)}(u)^{4 + \delta}]<C$; moreover, $\expv[\sup_{u} X_t(u)^2]<C$ as well as $\expv[(\xi^\smallO_k)^4]\leq C(\lambda^\smallO_k)^2$ for any $k \leq r$;
\item\label{ass:MomentsZ} $\expv[\sup_{u} Z_t(u)^2]<C$;
\item \label{ass:Lipschitz}  it holds that $\abs{\text{Cov}(X_t^{(d)}(u), X_t^{(d')}(v)) - \text{Cov}(X_t^{(d)}(u'), X_t^{(d')}(v'))} \leq C \, \norm{(u,v) - (u',v')}$ where~$\norm{(u,v) - (u',v')}^2 = (u-u')^2 + (v-v')^2$;
\item \label{ass:Eigenfunctions} $\sup_{k \leq r} \sup_{u} \abs{{\varphi_k^{\smallO}}^{(d)}(u)} < C$;
\item \label{ass:Errors} the errors $(e^{(d)}_t: t \leq T)$ are i.i.d.~$N$-variate random variables with zero mean, independent of the sequence $(\mathcal{X}_t: t \leq T)$, and satisfy
\begin{enumerate}
\item $\lambda_{\text{max}}(\mathbbm{E}_\smallO\mathbbm{E}_\smallO') = O_p(\max\{N, T\})$,
\item $\sum_{i,j=1}^N \abs{\expv[e_{ti}^{(d)}e^{(d)}_{tj}]} = O(N)$,
\item $\max_{i\leq N} \expv[(e^{(0)}_{ti})^2] < C$;
\end{enumerate}\newpage
\item \label{ass:Tail} there exist $b>0$ and $\gamma_2 > 0$ such that $\gamma = (1/\gamma_1 + 3/\gamma_2)^{-1} < 1$ and for any $\epsilon > 0$,
\begin{enumerate}
\item $\sup_{k\leq r}\prob(\abs{\xi^\smallO_{tk}/\sqrt{\lambda^\smallO_k}}>\epsilon) \leq \exp(-(\epsilon/b)^{\gamma_2})$;
\item $\sup_{u}\prob(\abs{Z_t(u)}>\epsilon) \leq \exp(-(\epsilon/b)^{\gamma_2})$;
\item $\max_{i\leq N}\prob(\abs{e^{(0)}_{ti}}>\epsilon) \leq \exp(-(\epsilon/b)^{\gamma_2})$;
\end{enumerate}
\item \label{ass:Asymptotics} $r/\lambda^\smallO_{r} = o(\min\{\sqrt{N},\sqrt{T/\log(N)}\})$ and $\log(N) = O(T^{\gamma/(2-\gamma)})$.
\end{enumerate}

The definition of $\alpha$-mixing coefficients is given in Section~\ref{sec:Mixing} of the supplementary file.  Assumptions~\ref{ass:Mixing} and \ref{ass:Tail} are similarly considered in \cite{fan2011high,fan2013large} and needed for applying a Bernstein inequality in the proof of Theorem~\ref{thm:Convergence}. 
Moment conditions in Assumptions~\ref{ass:MomentsX} and~\ref{ass:MomentsZ} are satisfied in the special case of Gaussian processes with continuous sample paths (see \cite{landau1970supremum}). Assumption~\ref{ass:Lipschitz} is a Lipschitz condition on the covariance function and implies mean-square continuity. Stronger smoothness conditions such as differentiability are not imposed. Particular examples of processes with Lipschitz continuous covariance functions are given by the Brownian motion or an Ornstein-Uhlenbeck process. Similar to \cite{kneip2020optimal}, we assume bounded eigenfunctions in Assumption~\ref{ass:Eigenfunctions} which holds for a Fourier basis of sine and cosine functions. Consistency of estimators could also be shown if~$\sup_{k \leq r} \sup_{u} \abs{{\varphi_k^{\smallO}}^{(d)}(u)} \to \infty$ as $r\to\infty$ but convergence rates become slower. The conditions on the noise term in Assumption~\ref{ass:Errors} are similar to \cite{bai2021approximate}.  Assumption~\ref{ass:Errors}(a) is automatically satisfied in the special case where $e^{(d)}_{ti}$ are i.i.d.~(across time-section, cross-section, and~$d$) with bounded fourth moments (see Theorem~2 in \cite{latala2005some} and \cite{moon2017dynamic} for a further discussion). In general, Assumption~\ref{ass:Errors} also allows weak dependence among the errors in the cross-section and across $d$. Assumption~\ref{ass:Asymptotics} puts restrictions on the joint asymptotic behavior of $r$, $N$, and~$T$. The second rate ensures that $N$ does not grow too fast compared to $T$.

\begin{theorem}\label{thm:Convergence}
Suppose that Assumptions~\ref{ass:Sample}-\ref{ass:Asymptotics} are satisfied. Then it holds that
\begin{equation}\label{eq:thmConvergence}
\sup_{u \in [0,1]} \, \abs{\widehat{\mathcal{L}(\mathcal{X}^\smallO_s)}(u)-\mathcal{L}(\mathcal{X}^\smallO_s)(u)} = O_p\left(\frac{r}{\lambda^\smallO_{r}}
\sqrt{\frac{1}{N} + \frac{\log(N)}{T}}\right),
\end{equation}
as $r, N, T \to \infty$.
\end{theorem}
Related to a discussion in \cite[p.140-141]{bai2003inferential}, we remark that the limit in Theorem~\ref{thm:Convergence} is simultaneous in the sense that $(r, N, T)$ is allowed to grow along all paths which satisfy the restrictions in Assumption~\ref{ass:Asymptotics}. The term $\sqrt{1/N + \log(N)/T}$ in~\eqref{eq:thmConvergence} is similarly derived by \cite{fan2013large} in approximate factor models with fixed $r$; the factor $(r/\lambda^\smallO_{r})$ reflects our assumption of a growing rank. In fact, the $r$-th eigenvalue $\lambda^\smallO_{r}$ measures pervasiveness of the weakest factor score. For this reason, a faster decay of eigenvalues entails a more difficult reconstruction problem and $\lambda^\smallO_{r}$ therefore enters~\eqref{eq:thmConvergence} in the inverse.

Theorem~\ref{thm:Convergence} includes the special case without additional covariates where $\mathcal{X}_s(u) = X_s(u)$ constitutes a univariate random function. Let us have a closer look at this case and compare Theorem~\ref{thm:Convergence} with the related Theorem~4.2 in \cite{kneip2020optimal}. From a theoretical perspective, \cite{kneip2020optimal} consider infinite dimensional functional data and derive a nonparametric estimator of the approximation $\mathcal{L}(X^\smallO_s)(u) = \sum_{k=1}^{r} \xi^\smallO_{sk}\tilde{\varphi}^\smallO_k(u)$ to $\sum_{k=1}^\infty \xi^\smallO_{sk}\tilde{\varphi}^\smallO_k(u)$ (employing our notation in the sequel). We note that in our setting $r$ represents the growing dimension of the function space, while it serves as a truncation parameter in \cite{kneip2020optimal}. Moreover, \cite{kneip2020optimal} assume almost surely twice continuously differentiable functions which helps to control the discretization error and is particularly useful in settings where the number of observation points $N$ is comparably small. Assuming that $N$ grows at least proportional to~$T$ and $\lambda^\smallO_k = k^{-\nu}$ for some $\nu > 1$, the estimator~$\widehat{\mathcal{L}_\text{KL}(X^\smallO_s)}$ of \cite{kneip2020optimal} satisfies $\abs{\widehat{\mathcal{L}_\text{KL}(X^\smallO_s)}(u) - \mathcal{L}(X_s^\smallO)(u)} = O_p(r^{\nu/2+5/2}\sqrt{1/T})$ for any fixed~$u \in [0,1]$. In contrast, our rate is uniform over $u\in [0,1]$. Inspection of the proof of Theorem 1 actually shows that the $\log(N)$ term in \eqref{eq:thmConvergence} is introduced by taking the supremum over $u \in [0,1]$. Without the supremum, \eqref{eq:thmConvergence} reduces in the case of relatively large $N$ to $\abs{\widehat{\mathcal{L}(X_s^\smallO)}(u)-\mathcal{L}(X_s^\smallO)(u)} = O_p(r^{\nu + 1}\sqrt{1/T})$. This shows that for large $N$, our rate improves \cite{kneip2020optimal} when rough functions with $\nu \in (1, 3)$ are considered.  

Our approach can also be applied if $\bigO = [0,1]$. It can then be used to recover the underlying signal $X_s(u)$ from the discrete and noisy observations in \eqref{eq:Measurements}. In fact, it holds $\mathcal{L}(\mathcal{X}^\smallO_s)(u) = X_s(u)$ for any $u \in [0,1]$ and we denote the corresponding estimator $\widehat{X}_s(u) = \widehat{\mathcal{L}(\mathcal{X}^\smallO_s)}(u)$. The following corollary is immediate.

\begin{corollary}\label{cor:Convergence}
Suppose that Assumptions~\ref{ass:Sample}-\ref{ass:Asymptotics} hold for $\bigO = [0,1]$. Then it holds that
\begin{equation*}
\sup_{u \in [0,1]} \, \abs{\widehat{X}_s(u)-X_s(u)} = O_p\left(\frac{r}{\lambda^\smallO_{r}}
\sqrt{\frac{1}{N} + \frac{\log(N)}{T}}\right),
\end{equation*}
as $r, N, T \to \infty$.
\end{corollary}

The above result is of independent interest for preprocessing noisy functional data and should be compared to a corresponding rate derived in \cite{hormann2022consistently}. While focusing on uniform convergence rates for the maximum over $s\leq T$, the authors show in our notation that $\max_{i \leq N} \abs{\widehat{X}_s(u_i) - X_s(u_i)} = O_p(1/T^{1/4} + T^{1/4}/\sqrt{N})$ (for simplicity, we consider the rank to be fixed). As indicated by Corollary~\ref{cor:Convergence}, careful analysis of the involved quantities allows us to improve this bound to $O_p(\sqrt{1/N + \log(N)/T})$.
\section{Simultaneous prediction bands}\label{sec:PredictionBands}

Once convergence of the reconstruction estimator is understood, questions concerning inferential procedures arise. In practical applications, prediction bands are useful to assess the reconstruction accuracy. A method for constructing such bands in the context of partially observed functional data has been suggested by \cite{kraus2015components}. However, as Kraus' method assumes continuous measurements with no errors, the construction of prediction bands in the noisy observation regime~\eqref{eq:Measurements} remains an open question and is addressed in the following. The prediction bands we consider are of the form
\begin{equation*}
\widehat{\mathcal{L}(\mathcal{X}_s^\smallO)}(u)\pm q\times \omega(u), \qquad u\in\mathrm{M},
\end{equation*}
where $\omega: \bigM \to \mathbb{R}^+$ is some limiting function and $q$ a scalar which determines the width of the prediction band. To incorporate special characteristics of the curve, we choose $\omega(u) = \text{sd}(Z(u))$, noting that the proposed procedure can be easily adapted to other limiting functions as well. Throughout this section, we require a uniformly consistent estimator $\widetilde{X}_t$ of completely observable curves ($t \in \mathcal{T}$) as stated below.

\begin{enumerate}[label = (A\arabic*)]
\setcounter{enumi}{11}
\item \label{ass:Bands} It holds that $\sup_{u \in \bigM} \abs{\widetilde{X}_t(u) - X_t(u)} = o_p(1)$ for $t \in \mathcal{T}$.
\end{enumerate}

Since we do note impose additional conditions other than consistency in Assumption~\ref{ass:Bands}, the procedure is flexible and combinable with a wide selection of smoothers. If the assumptions of Corollary~\ref{cor:Convergence} are satisfied, one may for instance use our factor-based estimator~$\widehat{X}_t$. To approximate the limiting function~$\omega(u) = \text{sd}(Z(u))$, we define
\begin{align*}
\widehat{\omega}(u) = \sqrt{\frac{1}{T_\smallC} \sum_{t \in \mathcal{T}} (\widehat{Z}_t(u) - \mean{Z}(u) )^2}, && \mean{Z}(u) = \frac{1}{T_\smallC} \sum_{t \in \mathcal{T}} \widehat{Z}_t(u), \qquad u \in \bigM,
\end{align*}
where $\widehat{Z}_t(u) = \widetilde{X}_t(u) - \widehat{\mathcal{L}(\mathcal{X}_t^\smallO)}(u)$. Under our assumptions, one can show that $\sup_{u \in \bigM} \abs{\widehat{\omega}(u) - \omega(u)} = o_p(1)$. Simultaneous prediction bands are then based on the following result. 
\begin{theorem}\label{thm:PredictionBands}
Let $\alpha \in (0,1)$ and suppose that Assumptions~\ref{ass:Sample}-\ref{ass:Bands} are satisfied. If $\omega$ is bounded from above and below,
\begin{equation*}
\lim \, \prob\left(\sup_{u \in \bigM}\, \frac{ \abs{\widehat{\mathcal{L}(\mathcal{X}_s^\smallO)}(u)-X_s(u)}}{\widehat{\omega}(u)} > q_\alpha \right) \leq \alpha,
\end{equation*}
where $q_\alpha$ is the $(1-\alpha)$-quantile of  the variable $\zeta_t = \sup_{u \in \bigM} \{\abs{Z_t(u)}/\omega(u)\}$ and the limit is taken as~$r, N, T \to \infty$.
\end{theorem}

The above result is based on the unknown value of $q_\alpha$. \cite{kraus2015components} proposes to estimate $q_\alpha$ in a similar setting by sampling from Gaussian processes. Here, we pursue a different approach and define for $t \in \mathcal{T}$,
\begin{equation*}
\widehat{\zeta}_t = \sup_{u \in \bigM} \frac{\abs{\widehat{\mathcal{L}(\mathcal{X}_t^\smallO)}(u)-\widetilde{X}_t(u)}}{\widehat{\omega}(u)}.
\end{equation*}
We then estimate $q_\alpha$ by the empirical $(1-\alpha)$-quantile $\widehat{q}_\alpha$ of $(\widehat{\zeta}_t: t \in \mathcal{T})$.

\begin{proposition}\label{prop:Quantile}
Let the assumptions of Theorem~\ref{thm:PredictionBands} be satisfied and suppose that the cdf $F_\zeta$ of $\zeta$ is continuous. Then, $F_\zeta(\widehat{q}_\alpha) \to 1-\alpha$ in probability as~$r, N, T \to \infty$.
\end{proposition}

Using Theorem~\ref{thm:PredictionBands} and Proposition~\ref{prop:Quantile}, we are finally able to state simultaneous prediction bands,
\begin{equation}\label{eq:PredBands}
\widehat{\mathcal{L}(\mathcal{X}_s^\smallO)}(u) \pm \widehat{q}_\alpha \times \widehat{\omega}(u), \qquad u \in \bigM.
\end{equation}
The proposed procedure is easy to implement and does not require computationally expensive bootstrap samples. We study numerical properties in Section~\ref{sec:AdequacyOfPredictionBands}.

\section{Finite sample properties}\label{sec:FiniteSampleProps}

We consider the following processes,
\begin{equation*}
X_t(u) = \mu(u) + \sum_{k=1}^{50} \xi^{(0)}_{tk} \phi^{(0)}_k(u), \quad X^{(1)}_t(u) = \int_0^1 \beta(u,v)X_t(v)\, \dv, \quad u \in [0,1],
\end{equation*}
where we define $\mu(u) = \sin(\pi u)$, $\phi^{(0)}_{2\ell - 1}(u) = \sqrt{2}\sin(2\ell \pi u)$, as well as $\phi^{(0)}_{2\ell}(u) = \sqrt{2}\cos(2\ell \pi u)$ for any $\ell = 1, \dots, 25$. Furthermore, we consider i.i.d.~standard normal distributed variables $\eta_{tk} \sim N(0,1)$ and set $\xi^{(0)}_{tk} = (\lambda_k^{(0)})^{1/2}\eta_{tk}$. We study both an exponential decay of eigenvalues, $\lambda_k^{(0)} = e^{-k}$, and a polynomial decay, $\lambda_k^{(0)} = k^{-3}/2$. Concerning the covariate, set $\beta(u,v) = \sum_{k,\ell=1}^2 b_{k\ell} \phi^{(0)}_{k}(u)\phi^{(0)}_{\ell}(v)$ for $(b_{11}, b_{12}, b_{21}, b_{22}) = (1.1, 0.7, 0.5, 0.3)$. The processes are recorded on a grid $0 = u_1 < u_2 < \dots < u_N = 1$ of $N = 51$ equispaced points and measured with error,
\begin{equation*}
y_{ti}^{(d)} = x_{ti}^{(d)} + e_{ti}^{(d)}, \qquad d \in \{0,1\},
\end{equation*}
where $e_{ti}^{(d)}$ constitute i.i.d.~$N(0,\sigma_e^2)$ distributed random variables. We suppose that $T_\smallC \in \{50, 100\}$ training targets are completely observable whereas 50 test targets, indexed by $\ell = 1, \dots, 50$, are only observable on $\bigO_\ell = [0, D_\ell]\subset [0,1]$. Define $D_\ell \sim U(1/2, 3/4)$ (Setting~A) and $D_\ell \sim U(1/4, 3/4)$ (Setting~B). The subintervals $\bigO_\ell$ in Setting~A have mean length $5/8$ and are generally larger than the sets in Setting~B with length $1/2$. To assess empirical performance of different reconstruction procedures, we generate samples in $B = 100$ simulation runs indexed by $b = 1, \dots, B$. The \textit{maximum absolute error} (MAE$_b$) in run $b$ is then estimated by
\begin{equation*}
\text{MAE}_b = \frac{1}{50} \sum_{\ell =1}^{50} \max_{i \leq N} \, \abs{X_{b,\ell}(u_i) - \widehat{X}_{b,\ell}(u_i)},
\end{equation*}
where $\widehat{X}_{b,\ell}$ denotes some reconstruction of $X_{b,\ell}$ computed on the sample $b$. We further define
\begin{equation*}
\text{MAE} = \frac{1}{B} \sum_{b=1}^B \text{MAE}_b, \qquad \text{SD} = \sqrt{\frac{1}{B} \sum_{b=1}^B (\text{MAE}_b-\text{MAE})^2},
\end{equation*}
measuring the mean and standard deviation of maximum absolute errors over all simulation runs.

All simulations are carried out in R (\cite{R}) using the packages $\texttt{ReconstPoFD}$ of \cite{reconstpofd} and \texttt{fdapace} of \cite{fdapace}. The code can be downloaded from the first author's GitHub repository \texttt{fdReconstruct}.

\subsection{Choice of parameters}\label{sec:ChoiceOfParameters}

The estimator of the best reconstruction $\widehat{\mathcal{L}(\mathcal{X}^{\smallO})}$ in~\eqref{eq:Estimator} depends on an unknown rank $r$ as well as the choice of weights $w_d$. Below, we discuss empirical approaches.

\subsubsection{Choice of weights}

As noted by \cite{happ2018multivariate}, a suitable choice of the weights is particularly important in the case where the components of $\mathcal{X}$ differ in ranges (for instance if they are measured in different units). We follow their empirical approach and set
\begin{equation*}
w_0 = \left(\int_0^1 \widehat{\var}(X(u)) \,\du \right)^{-1}, \qquad w_1 = \left(\int_0^1 \widehat{\var}(X^{(1)}(u)) \,\du \right)^{-1},
\end{equation*}
where the variance is estimated from the completely observable part of the data. This ensures that all components contribute equally to the total variation (\cite{happ2018multivariate}).

\subsubsection{Choice of rank}

Once the weights are specified, we compute $\bigO$-specific numbers of factors using 5-fold cross validation on the completely observable subsample. An implementation of the cross-validation procedure is shown in Algorithm~\ref{alg:cross-valitation}.

\begin{algorithm}
\caption{$K$-fold cross-validation for choice of rank $r$}\label{alg:cross-valitation}
\begin{algorithmic}[1]
\State Select $\bigO$ and $r_\text{max}$;
\State Split $\mathcal{T}$ into $K$ index sets $\mathcal{T}_1, \mathcal{T}_2, \dots, \mathcal{T}_K$ of approximately equal size;
\State \textbf{For} $r = 1, \dots, r_\text{max}$:
\State \indent $\text{SSE}(r) \leftarrow 0$;
\State \indent \textbf{For} $k = 1, \dots, K$:
\State \indent \indent \textbf{For} $s$ in $\mathcal{T}_k$:
\State \indent \indent \indent Compute the rank-$r$ reconstruction $\widehat{\mathcal{L}_r(\mathcal{X}_s^\smallO)}$ of the pseudo-observed part $\mathbbm{y}_{\smallO, s}$ \indent \indent \indent using the subsample $(\mathbbm{y}_t: t \in \mathcal{T}\setminus \mathcal{T}_k)$;
\State \indent \indent \indent $\text{SSE}(r) \leftarrow \text{SSE}(r) + \sum_{i=1}^N (\widehat{\mathcal{L}_r(\mathcal{X}_s^\smallO)}(u_i) - \mathbbm{y}_{si})^2\mathbbm{1}\{u_i \in [0,1]\setminus \bigO\}$;
\State \textbf{Return} $\widehat{r} = \arg \min_r\, \text{SSE}(r)$.
\end{algorithmic}
\end{algorithm}

\subsection{Comparison of reconstruction procedures}

Next, we compare our proposed methodology to alternative procedures. Concerning our factor-based estimator~\eqref{eq:Estimator}, we consider both the univariate version~(\textsc{Uni}), where the covariate $X^{(1)}$ is ignored, as well as the multivariate version (\textsc{Mult}), which comprises information of $X$ and~$X^{(1)}$. Second, we study the linear reconstruction procedure for univariate functional data proposed by \cite{kneip2020optimal} which is specifically intended for the partial observation regime as in the underlying work. In fact, the reconstruction procedure of \cite{kneip2020optimal} is very similar to our method for the univariate case since both are based on the linear reconstruction problem~\eqref{eq:MinimizationProblem}. However, unlike our method, \cite{kneip2020optimal} use smoothing and thus require smooth functions. The procedure comes with different options. More specifically, the authors suggest to either estimate the function's observed and unobserved part at once using a linear reconstruction estimator, or perform smoothing on the observed part and use the reconstruction estimator solely for the missing part. In the latter case, some alignment is required to ensure continuous reconstructions. Moreover, the authors estimate scores $\xi_k^\smallO$ either by approximating integrals or using conditional expectations (CE). We restrict our attention to the \textsc{AYesCE} (with alignment and CE scores) option which has shown good performance in simulations of \cite{kneip2020optimal}.  Third, we examine the \textsc{Pace} method of \cite{yao2005functionalb} which in principle assumes sparse functional data and recovers functions from irregular observations inside the interval~$[0,1]$. As pointed out by \cite{kneip2020optimal}, \textsc{Pace} differs fundamentally from the considered reconstruction procedures, as it approximates the expansion $\sum_k \expv[\xi_k\vert X^\smallO] \phi_k(u)$ instead of the linear operator $\mathcal{L}(\mathcal{X}^\smallO)(u) = \sum_{k} \xi^\smallO_k \tilde{\phi}^\smallO_k(u)$. We note that a straightforward approximation of the Karhunen-Loève decomposition $X(u) = \sum_k \xi_k \phi_k(u)$ is not possible in the partial observation regime, as the ``global'' scores $\xi_k$ are not identified in this case. The fourth procedure in our simulation study is based on the function-on-function regression model $\expv[X(u)\vert X^{(1)}] = a(u) + \int_0^1 b(u,v) X^{(1)}(v)\, \dv$ of \cite{yao2005functionalb}. Here, $a$ and $b$ are estimated from the completely observable functions and $X_s$ is then reconstructed by $\widehat{X}_s = \widehat{a}(u) + \int_0^1 \widehat{b}(u,v) X_s^{(1)}(v) \, \dv$.

Reconstructions of two particular curves using our factor-based procedure are displayed in Figure~\ref{fig:simulationexamples}. The main features of the missing trajectories are reconstructed arguably well, hidden features are covered by the simultaneous prediction bands. Table~\ref{tab:simres} summarizes the results of the comparative simulation study. As it can be seen, the factor-based procedures generally perform better than the procedures involving smoothing. Our proposed multivariate reconstruction method \textsc{Mult} benefits from the additional covariate and outperforms its competitors in every scenario. The univariate version \textsc{Uni} (like \textsc{Pace} and \textsc{AYesCE}) is solely based on $X$ and performs second while \textsc{PACE} ranks third ahead of \textsc{AYesCE}. The benchmark model \textsc{Flm}, which ignores information of incomplete curves, exhibits the worst performance. Except for \textsc{Flm}, errors are generally higher in Setting B, where the observable sets are smaller. Factor-based procedures clearly profit by lower noise variance which is less apparent for the other approaches.

\begin{figure}
\centering
\includegraphics[width=\linewidth]{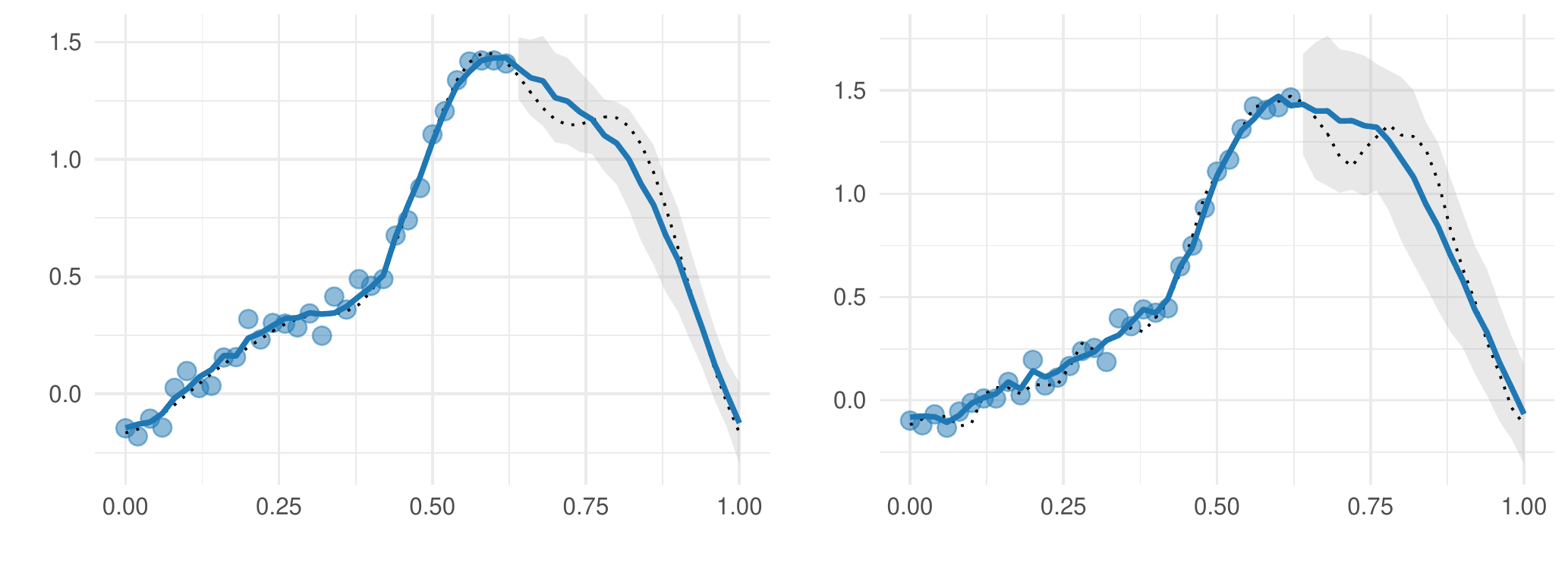}
\caption{Reconstructions of incomplete curves for $\lambda_k^{(0)} = e^{-k}$ (left) and  $\lambda_k^{(0)} = k^{-3}/2$ (right) in Setting A with $\sigma_e = 0.05$ and $T_\smallC = 100$ using additional covariate information (\textsc{Mult}). Solid lines refer to factor-based reconstructions, dotted lines denote true targets, and points noisy measurements. The shaded regions correspond to 95\% prediction bands (Section~\ref{sec:PredictionBands}).}
\label{fig:simulationexamples}
\end{figure}

\begin{table}
\caption{\text{MAE} and \text{SD} (in parentheses) of different reconstruction procedures. The factor-based procedures proposed in this work are labeled \textsc{Uni} (without covariate) and \textsc{Mult} (with covariate). \textsc{PACE} refers to the smoothing procedure of \cite{yao2005functional}, \textsc{AYesCE} to \cite{kneip2020optimal}, and \textsc{FLM} to the function-on-function regression model of \cite{yao2005functionalb}.}\label{tab:simres}
\centering
\begin{tabular}{|cccc|ccccc|}
\hline
 & $\lambda_k^{(0)}$ & $\sigma_e$ & $T_\smallC$ & \textsc{Uni} & \textsc{Mult} & \textsc{Pace} & \textsc{AYesCE} & \textsc{Flm} \\ \hline
A & $e^{-k}$           &  0.1 &  50 & 0.26 (0.02) & 0.23 (0.02) & 0.41 (0.05) & 0.44 (0.03) & 0.84 (0.05) \\
  &                    &      & 100 & 0.24 (0.02) & 0.21 (0.01) & 0.32 (0.04) & 0.40 (0.03) & 0.82 (0.05) \\
  &                    & 0.05 &  50 & 0.19 (0.01) & 0.16 (0.01) & 0.41 (0.07) & 0.42 (0.04) & 0.84 (0.06) \\
  &                    &      & 100 & 0.17 (0.01) & 0.14 (0.01) & 0.28 (0.05) & 0.39 (0.04) & 0.81 (0.05) \\
  & $\frac{k^{-3}}{2}$ & 0.1  &  50 & 0.33 (0.02) & 0.30 (0.01) & 0.46 (0.06) & 0.47 (0.04) & 0.64 (0.03) \\
  &                    &      & 100 & 0.30 (0.01) & 0.27 (0.01) & 0.39 (0.05) & 0.44 (0.03) & 0.63 (0.04) \\
  &                    & 0.05 &  50 & 0.29 (0.02) & 0.24 (0.01) & 0.44 (0.05) & 0.45 (0.04) & 0.63 (0.04) \\
  &                    &      & 100 & 0.27 (0.02) & 0.22 (0.01) & 0.37 (0.05) & 0.43 (0.03) & 0.63 (0.03) \\
\hline
B & $e^{-k}$           & 0.1  &  50 & 0.39 (0.03) & 0.33 (0.03) & 0.51 (0.05) & 0.54 (0.05) & 0.83 (0.05) \\
  &                    &      & 100 & 0.36 (0.03) & 0.30 (0.02) & 0.42 (0.04) & 0.51 (0.04) & 0.83 (0.05) \\  &                    & 0.05 &  50 & 0.32 (0.03) & 0.26 (0.03) & 0.51 (0.06) & 0.53 (0.06) & 0.83 (0.06) \\
  &                    &      & 100 & 0.30 (0.03) & 0.24 (0.02) & 0.38 (0.04) & 0.50 (0.04) & 0.83 (0.05) \\
  & $\frac{k^{-3}}{2}$ & 0.1  &  50 & 0.40 (0.02) & 0.36 (0.02) & 0.49 (0.04) & 0.54 (0.04) & 0.64 (0.04) \\
  &                    &      & 100 & 0.38 (0.02) & 0.33 (0.02) & 0.44 (0.04) & 0.51 (0.03) & 0.63 (0.03) \\
  &                    & 0.05 &  50 & 0.37 (0.03) & 0.31 (0.02) & 0.48 (0.05) & 0.53 (0.04) & 0.64 (0.04) \\
  &                    &      & 100 & 0.35 (0.02) & 0.29 (0.02) & 0.42 (0.05) & 0.50 (0.03) & 0.62 (0.04) \\
\hline
\end{tabular}
\end{table}

\subsection{Adequacy of prediction bands}\label{sec:AdequacyOfPredictionBands}

For the prediction bands~\eqref{eq:PredBands}, we consider a smoother $\widetilde{X}_t$ based on cubic splines. To analyze the performance of simultaneous prediction bands, we simulate training and test data as before and compute the proportion of bands covering the corresponding test targets in each simulation run. Table~\ref{tab:bands} reports the mean coverage probabilities along with standard deviations, where the quantities are taken over estimated coverages in $B = 100$ simulation runs. 

\begin{table}
\caption{Mean coverage probabilities and standard deviations (in parentheses).}\label{tab:bands}
\centering
\begin{tabular}{|cccc|ccc|}
\hline
\multirow{2}{*}{Setting} & \multirow{2}{*}{$\lambda_k^{(0)}$} & \multirow{2}{*}{$\sigma_e$} & \multirow{2}{*}{$T_\smallC$} & \multicolumn{3}{|c|}{nominal coverage} \\ & & & & \multicolumn{1}{|c}{75\%} & \multicolumn{1}{c}{90\%} & \multicolumn{1}{c|}{95\%} \\ \hline
A &          $e^{-k}$ & 0.1  &  50 & 0.78 (0.06) & 0.88 (0.05) & 0.93 (0.04) \\
  &                   &      & 100 & 0.79 (0.05) & 0.92 (0.03) & 0.95 (0.03) \\
  &                   & 0.05 &  50 & 0.74 (0.07) & 0.87 (0.05) & 0.92 (0.04) \\
  &                   &      & 100 & 0.79 (0.05) & 0.91 (0.03) & 0.95 (0.03) \\
  & $\frac{k^{-3}}{2}$&  0.1 & 50  & 0.70 (0.08) & 0.85 (0.06) & 0.91 (0.05) \\
  &                   &      & 100 & 0.72 (0.06) & 0.88 (0.04) & 0.94 (0.03) \\
  &                   & 0.05 & 50  & 0.67 (0.08) & 0.83 (0.06) & 0.89 (0.05) \\
  &                   &      & 100 & 0.71 (0.06) & 0.87 (0.05) & 0.93 (0.03) \\ \hline
B &          $e^{-k}$ &  0.1 & 50  & 0.74 (0.07) & 0.88 (0.05) & 0.92 (0.04) \\
  &                   &      & 100 & 0.79 (0.06) & 0.91 (0.04) & 0.95 (0.03) \\
  &                   & 0.05 &  50 & 0.74 (0.07) & 0.87 (0.05) & 0.93 (0.04) \\
  &                   &      & 100 & 0.77 (0.06) & 0.90 (0.04) & 0.95 (0.03) \\
  & $\frac{k^{-3}}{2}$&  0.1 & 50  & 0.69 (0.07) & 0.84 (0.06) & 0.90 (0.05) \\
  &                   &      & 100 & 0.72 (0.05) & 0.87 (0.04) & 0.93 (0.03) \\
  &                   & 0.05 & 50  & 0.68 (0.07) & 0.83 (0.06) & 0.88 (0.05) \\
  &                   &      & 100 & 0.72 (0.06) & 0.87 (0.04) & 0.93 (0.03) \\
\hline
\end{tabular}
\end{table}

\section{Real data illustration}\label{sec:RealDataApplication}

We return to the temperature data set from the introduction to further illustrate our methodology. The data are provided by \cite{data2023} and consist of $N = 48$ half-hourly temperature values which were recorded between July~1 and September~14, 2022 in the east (Petersgasse, 362~m above sea level, referred to as Station {\tt E}) and west (Plabutsch, 754~m, referred to as Station {\tt W}) of Graz~(Austria). See Figure~\ref{fig:sample} for a plot of daily temperature curves. The overall mean of 21.6~°C in the east is one degree higher than the western mean of~20.6~°C. A correlation of around 90$\%$ between half-hourly temperatures in Stations {\tt E} and {\tt W} suggests to consider measurements from Station {\tt W} as an additional covariate for the reconstruction problem. Out of $T = 76$ curves, only~$T_\smallC = 66$ are observed at all discretization points in Station {\tt E}. Using our proposed methodology, we reconstruct the remaining incompletely observed curves. The results for two particular days are shown in Figure~\ref{fig:weatherexamples}, the other days are presented in the supplementary file. As it can be seen, our factor-based approaches (with and without covariate) lead to fairly similar results for August~17. However, the situation is quite different for August~18. Here, the multivariate model appears to recover some additional features of the missing signal from covariate data in Station {\tt W}. The difference can be explained by a storm which hit the region on August~18 during the missing segment of Station {\tt E}. This has caused a sudden drop in the temperature around~16:00. Without covariate information, the model predicts a steady decline from the last available temperature at around~11:00. The cross validated number of factors in the multivariate model ($\widehat{r} =  15$ for August 17, $\widehat{r} = 24$ for August 18) further stresses the relevance of a high rank in real data examples.

\section{Acknowledgments}

This research was partly funded by the Austrian Science Fund (FWF) [P 35520].

\begin{figure}
\centering
\includegraphics[width=\linewidth]{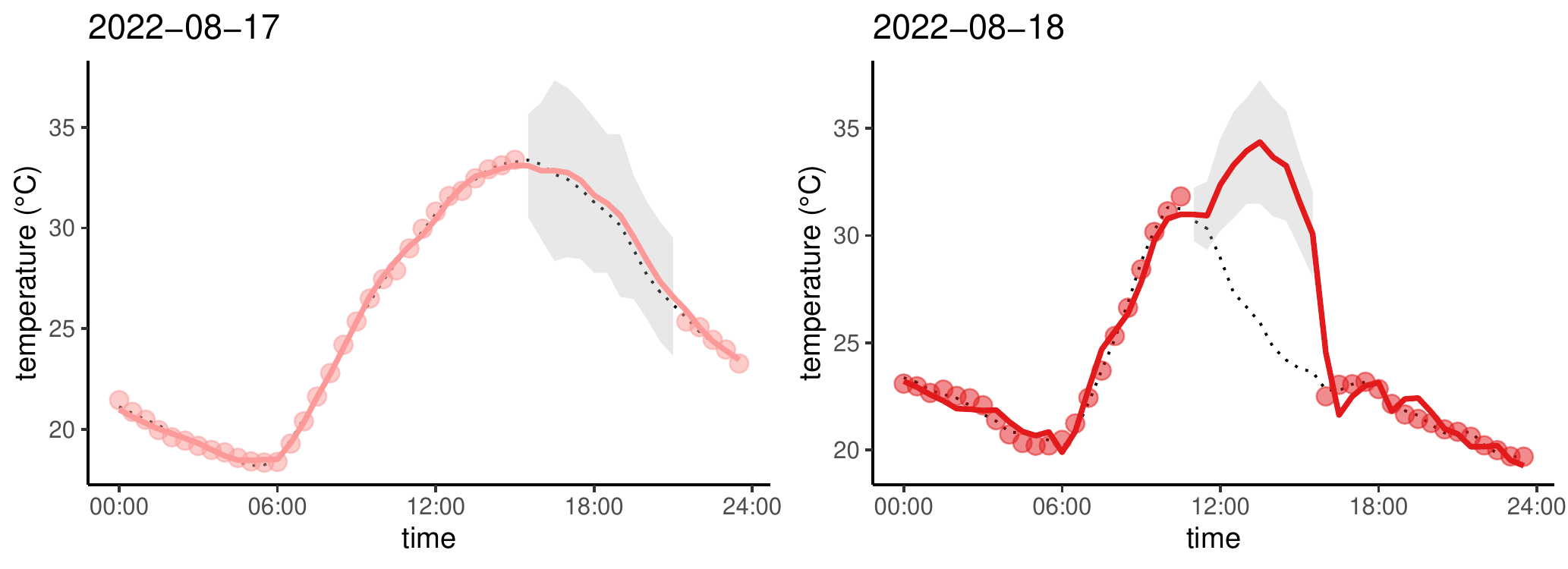}
\caption{Reconstructions of incompletely observed temperature curves recorded in the east of Graz (Austria) without covariate (dotted line) and western measurements as additional covariate (solid line). Points refer to noisy measurements; 95\% prediction bands were obtained from the procedure presented in Section~\ref{sec:PredictionBands} using additional covariate information and cubic splines for smoothing complete curves.}
\label{fig:weatherexamples}
\end{figure}

\bibliographystyle{apalike}
\bibliography{ref}

\newpage

\appendix

\begin{center}
  \LARGE Supplementary material for: \medskip
  
  Covariate-informed reconstruction of partially observed functional data via factor models
\end{center}
\vskip 1em
\begin{center}
  \large \lineskip .75em%
  \begin{tabular}[t]{c}
    Maximilian Ofner and Siegfried H\"ormann
  \end{tabular}\par
\end{center}
\begin{center}
  Graz University of Technology
\end{center}

\section{Main proofs}

In the following, we heavily rely on arguments from  \cite{bai2021approximate}. Regarding norms, let $\norm{\cdot}_2$ and $\normF{\cdot}$ denote the spectral and Frobenius norm of a matrix, respectively, and $\norm{\cdot}$ the euclidean norm of a vector. Recall that $\normS{\mathbbm{A}\mathbbm{x}} \leq \normS{\mathbbm{A}}\norm{\mathbbm{x}}, \normF{\mathbbm{A}\mathbbm{x}} \leq \normF{\mathbbm{A}}\norm{\mathbbm{x}},\normS{\mathbbm{A}\mathbbm{B}}\leq \normS{\mathbbm{A}}\normS{\mathbbm{B}},\normF{\mathbbm{A}\mathbbm{B}}\leq \normF{\mathbbm{A}}\normF{\mathbbm{B}}$ and $\normS{\mathbbm{A}} \leq \normF{\mathbbm{A}}$, for any matrix $\mathbbm{A}$, $\mathbbm{B}$ and vector $\mathbbm{x}$.

Additionally, we introduce the following notation. Let $$\mathbbm{f}_{\smallO,t} = (\xi^\smallO_{t1}/\sqrt{\lambda^\smallO_1}, \dots, \xi^\smallO_{tr}/\sqrt{\lambda^\smallO_r}),$$ be the $t$-th row of $\mathbbm{F}_\smallO$ and $$\bblambda^*_{\smallO,i} = (\sqrt{\lambda^\smallO_1}\tilde{\varphi}^\smallO_1(u_i), \dots, \sqrt{\lambda^\smallO_r}\tilde{\varphi}^\smallO_r(u_i)).$$ Consider the estimators $\widehat{\mathbbm{f}}_{\smallO,s} = \frac{1}{\sqrt{N}} \mathbbm{y}_{\smallO,s} \mathbbm{V}_\smallO \mathbbm{D}_\smallO^{-1} = \frac{1}{N} \mathbbm{y}_{\smallO,s} \widehat{\bbLambda}_\smallO \mathbbm{D}_\smallO^{-2}$ and $\widehat{\bblambda}^*_{\smallO,i} = \frac{1}{T_\smallC}\mathbbm{y}_{\smallC,i}' \widehat{\mathbbm{F}}_\smallO$, where $\widehat{\bbLambda}_\smallO = \sqrt{N} \mathbbm{V}_\smallO \mathbbm{D}_\smallO$ is an estimator for $\bbLambda_\smallO$.
Rewrite $\mathcal{L}(\mathcal{X}_s^\smallO)(u_i)$ in \eqref{eq:Operator} as
\begin{equation}\label{eq:OperatorMatrixNotation}
\mathcal{L}(\mathcal{X}_s^\smallO)(u_i) = \sum_{k=1}^{r} \xi^\smallO_{sk}\tilde{\varphi}^\smallO_k(u_i) = \bblambda^{*}_{\smallO,i} \mathbbm{f}_{\smallO,s}'.
\end{equation}

\subsection{Proof of Theorem 1}

To prove Theorem 1, we need to recover $\mathbbm{f}_{\smallO,s}$ and $\bblambda_{\smallO,i}^*$ from~$\mathbbm{Y}$ which is only possible up to rotation. Let us define
\begin{itemize}
\item $\mathbbm{G}_{\smallO} = \bbLambda_\smallO'\widehat{\bbLambda}_\smallO (\widehat{\bbLambda}_\smallO'\widehat{\bbLambda}_\smallO)^{-1}  = \frac{\bbLambda_\smallO'\widehat{\bbLambda}_\smallO}{N} \mathbbm{D}_\smallO^{-2}$,
\item$\mathbbm{H}_{\smallO} = (\widehat{\mathbbm{F}}_\smallO'\mathbbm{F}_\smallO)^{-1}\widehat{\mathbbm{F}}_\smallO' \widehat{\mathbbm{F}}_\smallO = (\frac{\widehat{\mathbbm{F}}_\smallO'\mathbbm{F}_\smallO}{T_\smallC})^{-1}$;
\end{itemize}
the involved quantities have already been introduced before. As can be seen, $\mathbbm{G}_{\smallO}$ results from the regression of columns of $\bbLambda_\smallO$ onto columns of $\widehat{\bbLambda}_\smallO$. In the same vein, $\mathbbm{H}_{\smallO}^{-1}$ resembles a regression of $\mathbbm{F}_\smallO$ onto $\widehat{\mathbbm{F}}_\smallO$. Appendix~\ref{sec:Equivalence} shows that the matrices $\mathbbm{G}_\smallO$ and $\mathbbm{H}_\smallO$ are asymptotically equivalent in the sense that $\normS{\mathbbm{G}_{\smallO}-\mathbbm{H}_{\smallO}} =  o_p(1)$. See also Section 3.2~in \citeAppendix{bai2021approximate}  for a discussion on equivalence of rotation matrices. The proofs of the following lemmas are given in Appendix~\ref{sec:LemmaRateF} and \ref{sec:LemmaRateL}, respectively.

\begin{lemma}\label{lemma:RateF}
Under the assumptions of Theorem~1, for every $s \leq T$,
\begin{equation*}
\norm{\widehat{\mathbbm{f}}_{\smallO,s}-\mathbbm{f}_{\smallO,s}\mathbbm{G}_{\smallO}} = O_p\left(\frac{1}{\lambda^\smallO_{r}}\sqrt{\frac{1}{N}+\frac{1}{T}}\right).
\end{equation*}
\end{lemma}

\begin{lemma}\label{lemma:RateLambda}
Under the assumptions of Theorem~1,
\begin{equation*}
\max_{i\leq N} \, \norm{\widehat{\bblambda}^*_{\smallO,i} - \bblambda^*_{\smallO,i}(\mathbbm{H}_{\smallO}')^{-1}} = O_p\left(
\sqrt{\frac{r}{\lambda^\smallO_{r}}\left(\frac{1}{N} + \frac{\log(N)}{T}\right)}\right).
\end{equation*}
\end{lemma}

\begin{proof}[Proof of Theorem 1]
We deduce from decomposition \eqref{eq:OperatorMatrixNotation} of $\mathcal{L}(\mathcal{X}_s^\smallO)$,
\begin{align*}
&\widehat{\mathcal{L}(\mathcal{X}_s^\smallO)}(u_i)-\mathcal{L}(\mathcal{X}_s^\smallO)(u_i) = \widehat{\bblambda}^*_{\smallO,i}\widehat{\mathbbm{f}}_{\smallO,s}'  - \bblambda^*_{\smallO,i}\mathbbm{f}_{\smallO,s}'\\
&\qquad = (\widehat{\bblambda}^*_{\smallO,i}-\bblambda^*_{\smallO,i} (\mathbbm{H}_{\smallO}')^{-1})
(\mathbbm{f}_{\smallO,s}\mathbbm{H}_{\smallO})'+\widehat{\bblambda}^*_{\smallO,i} (\widehat{\mathbbm{f}}_{\smallO,s}-\mathbbm{f}_{\smallO,s}\mathbbm{G}_{\smallO})' + \widehat{\bblambda}^*_{\smallO,i}(\mathbbm{G}_{\smallO}-\mathbbm{H}_{\smallO})'\mathbbm{f}_{\smallO,s}'.
\end{align*}
Taking norms, we observe
\begin{equation*}
\max_{i\leq N} \, \abs{\widehat{\mathcal{L}(\mathcal{X}_s^\smallO)}(u_i)-\mathcal{L}(\mathcal{X}_s^\smallO)(u_i)} \leq a + b + c,
\end{equation*}
where
\begin{align*}
&a = \max_{i\leq N} \,\norm{\widehat{\bblambda}^*_{\smallO,i}-\bblambda^*_{\smallO,i} (\mathbbm{H}_{\smallO}')^{-1}}\norm{\mathbbm{f}_{\smallO,s}}\normS{\mathbbm{H}_{\smallO}} && = O_p\left(\frac{{r}}{\lambda^\smallO_{r}}
\sqrt{\frac{1}{N} + \frac{\log(N)}{T}}\right),\\
& b = \max_{i\leq N}\,\norm{\widehat{\bblambda}^*_{\smallO,i}} \norm{\widehat{\mathbbm{f}}_{\smallO,s}-\mathbbm{f}_{\smallO,s}\mathbbm{G}_{\smallO}} && = O_p\left(\frac{1}{\lambda^\smallO_{r}}\sqrt{\frac{1}{N}+\frac{1}{T}}\right),\\
&c = \max_{i \leq N}\, \norm{\widehat{\bblambda}^*_{\smallO,i}}\normS{\mathbbm{G}_{\smallO}-\mathbbm{H}_{\smallO}}\norm{\mathbbm{f}_{\smallO,s}} && = O_p\left(\frac{\sqrt{{r}}}{\lambda^\smallO_{r}}\sqrt{\frac{{r}}{N} + \frac{1}{T}}\right).
\end{align*}
The rates follow from Lemma~\ref{lemma:RateF}, Lemma~\ref{lemma:RateLambda}, Lemma~\ref{lemma:H}\textit{(ii)}, Lemma~\ref{lemma:Rotation} and Lemma~\ref{lemma:Loadings}\textit{(ii)}, proven in the Appendix, as well as the fact $\norm{\mathbbm{f}_{\smallO,s}} = O_p(\sqrt{{r}})$. We see that the first term dominates the others. As a next step, we treat the discretization error. To this end, observe that $\expv[(X(u)-X(v))^2] \leq C \abs{u-v}$ by Assumption~\ref{ass:Lipschitz}. Thus,
\begin{equation}\label{eq:HoelderL}
\begin{aligned}
\abs{\mathcal{L}(\mathcal{X}_s^\smallO)(u)-\mathcal{L}(\mathcal{X}_s^\smallO)(v)} &\leq \sum_{k=1}^{r} \abs{\xi^\smallO_{sk}}\frac{\expv{[\abs{X(u)-X(v)}\abs{\xi^\smallO_k}}]}{\lambda^\smallO_k} \\
&\leq \sum_{k=1}^{r} \frac{\abs{\xi^\smallO_{sk}}}{\sqrt{\lambda^\smallO_k}}\expv{[(X(u)-X(v))^2]}^{1/2} \\
&\leq C \sum_{k=1}^{r} \frac{\abs{\xi^\smallO_{sk}}}{\sqrt{\lambda^\smallO_k}}\sqrt{\abs{u - v}} = O_p\left({r} \sqrt{\abs{u - v}}\right).
\end{aligned}
\end{equation}
Now, define $\mathcal{L}^*: \mathcal{H} \to L^2([0,1])$ to be the operator which results form $\mathcal{L}$ by linear interpolation, that is, for $f \in \mathcal{H}$ and $u \in [0,1]$,
\begin{equation}\label{eq:DiscOperator}
\begin{aligned}
\mathcal{L}^*(f)(u) &= \sum_{i = 1}^{N-1} \left(\mathcal{L}(f)(u_i)  + \frac{u - u_i}{u_{i+1}-u_i}(\mathcal{L}(f)(u_{i+1}) -\mathcal{L}(f)(u_i))   \right)\\
&\qquad \times \mathbbm{1} \{u \in [u_i,u_{i+1})\}.
\end{aligned}
\end{equation}
We get from~\eqref{eq:HoelderL},~\eqref{eq:DiscOperator}, and our assumption on the equidistant spacing of grid points,
\begin{equation*}
\begin{aligned}
&\sup_{u \in [0,1]} \, \abs{\mathcal{L}^*(\mathcal{X}_s^\smallO)(u)-\mathcal{L}(\mathcal{X}_s^\smallO)(u)} = \max_{i\leq N} \, \sup_{u \in [u_{i}, u_{i+1})}\, \abs{\mathcal{L}^*(\mathcal{X}_s^\smallO)(u)-\mathcal{L}(\mathcal{X}_s^\smallO)(u)} \\
&\quad \leq \max_{i\leq N} \, \left(\sup_{u \in [u_{i}, u_{i+1})}\, \abs{\mathcal{L}(\mathcal{X}_s^\smallO)(u_i)-\mathcal{L}(\mathcal{X}_s^\smallO)(u)} + \abs{\mathcal{L}(\mathcal{X}_s^\smallO)(u_{i+1})-\mathcal{L}(\mathcal{X}_s^\smallO)(u_i)}\right) \\
&\quad = O_p\left({r} \max_{i \leq N}\,  \sqrt{\abs{u_{i+1}-u_i}}\right)  = O_p\left(\frac{{r}}{\sqrt{N}}\right).
\end{aligned}
\end{equation*}
Combining the above, we conclude
\begin{align*}
&\sup_{u \in [0,1]} \, \abs{\widehat{\mathcal{L}(\mathcal{X}_s^\smallO)}(u)-\mathcal{L}(\mathcal{X}_s^\smallO)(u)} \\
&\qquad \leq \max_{i \leq N} \, \abs{\widehat{\mathcal{L}(\mathcal{X}_s^\smallO)}(u_i)-\mathcal{L}^*(\mathcal{X}_s^\smallO)(u_i)} + 
\sup_{u \in [0,1]} \, \abs{\mathcal{L}^*(\mathcal{X}_s^\smallO)(u)-\mathcal{L}(\mathcal{X}_s^\smallO)(u)}\\
&\qquad = O_p\left(\frac{{r}}{\lambda^\smallO_{r}}\sqrt{\frac{1}{N} + \frac{\log(N)}{T}}\right) + O_p\left(\frac{{r}}{\sqrt{N}}\right).
\end{align*}
The second term is dominated by the first.
\end{proof}

\subsection{Proof of Theorem~\ref{thm:PredictionBands}}

\begin{proof}
The definition $Z_t(u) = X_t(u) - \mathcal{L}(\mathcal{X}_t^\smallO)(u)$ yields

\begin{align*}
&\prob\left(\sup_{u \in \bigM}\, \frac{ \abs{X_t(u)-\widehat{\mathcal{L}(\mathcal{X}_t^\smallO)}(u)}}{\widehat{\omega}(u)} > q_\alpha \right)  \\
&\quad \leq \prob\left\lbrace \left(\sup_{u \in \bigM}\, \frac{ \abs{\mathcal{L}(\mathcal{X}_t^\smallO)(u)-\widehat{\mathcal{L}(\mathcal{X}_t^\smallO)}(u)}}{\omega(u)}+\sup_{u \in \bigM}\, \frac{ \abs{Z_t(u)}}{\omega(u)}\right)\sup_{u\in \bigM} \frac{\omega(u)}{\widehat{\omega}(u)} > q_\alpha\right\rbrace.
\end{align*}
Next, let us define
\begin{equation*}
A = \sup_{u \in \bigM}\, \frac{ \abs{\mathcal{L}(\mathcal{X}_t^\smallO)(u)-\widehat{\mathcal{L}(\mathcal{X}_t^\smallO)}(u)}}{\omega(u)},\quad B = \sup_{u \in \bigM}\, \frac{ \abs{Z_t(u)}}{\omega(u)}, \quad C = \sup_{u\in \bigM} \frac{\omega(u)}{\widehat{\omega}(u)}. 
\end{equation*}
For $\epsilon > 0$, we get
\begin{equation*}
\prob\left\lbrace (A + B)C > (\epsilon + q_\alpha)(1 + \epsilon)\right\rbrace \leq \prob(A > \epsilon) + \prob(B > q_\alpha) + \prob(C > 1+\epsilon).
\end{equation*}
By Theorem~\ref{thm:Convergence} and boundedness of $\omega(u)$, it holds $\prob(A > \epsilon) \to 0$. Furthermore, $\prob(B > q_\alpha) = \alpha$ by the definition of $q_\alpha$ and $\prob(C > 1+\epsilon) \to 0$ by our assumptions. Taking~$\epsilon \to 0$ concludes the proof of the theorem.
\end{proof}

\subsection{Proof of Proposition~\ref{prop:Quantile}}

\begin{proof}
Let $\widehat{F}_\zeta(z) = \frac{1}{T_\smallC} \sum_{t\in \mathcal{T}} \mathbbm{1}\{\zeta_t\leq z\}$ denote the ecdf of $\zeta_t$ and additionally define $\widetilde{F}_\zeta(z) = \frac{1}{T_\smallC} \sum_{t\in \mathcal{T}} \mathbbm{1}\{\widehat{\zeta}_t\leq z\}$. Observe $\abs{F_\zeta(\widehat{q}_\alpha)  - (1 - \alpha)} \leq \abs{F_\zeta(\widehat{q}_\alpha) - \widehat{F}_\zeta(\widehat{q}_\alpha)} + \abs{\widehat{F}_\zeta(\widehat{q}_\alpha) - \widetilde{F}_\zeta(\widehat{q}_\alpha)} + \abs{\widetilde{F}_\zeta(\widehat{q}_\alpha)  - (1 - \alpha)}$. Invoking the theorem of Glivenko--Cantelli, $\abs{F_\zeta(\widehat{q}_\alpha) - \widehat{F}_\zeta(\widehat{q}_\alpha)} = o_p(1)$. Moreover, using $\abs{\zeta_t-\widehat{\zeta}_t} = o_p(1)$ and continuity of the distribution function, it can be shown that $\sup_{z} \abs{\widehat{F}(z) - \widetilde{F}(z)} = o_p(1)$, implying $\abs{\widehat{F}_\zeta(\widehat{q}_\alpha) - \widetilde{F}_\zeta(\widehat{q}_\alpha)} = o_p(1)$. Finally, $\abs{\widetilde{F}_\zeta(\widehat{q}_\alpha)  - (1 - \alpha)} = o_p(1)$ by the definition of empirical quantiles. We thus conclude that $\abs{F_\zeta(\widehat{q}_\alpha)  - (1 - \alpha)} = o_p(1)$.
\end{proof}

\section{Factors: Proof of Lemma~\ref{lemma:RateF}}\label{sec:LemmaRateF}

\begin{proof}
Using $\widehat{\mathbbm{f}}_{\smallO,s}  = \frac{1}{N} \mathbbm{y}_{\smallO,s} \widehat{\bbLambda}_\smallO \mathbbm{D}_\smallO^{-2}$ and $\mathbbm{y}_{\smallO,s} = \mathbbm{f}_{\smallO,s} \bbLambda_\smallO' + \mathbbm{e}_{\smallO,s}$, we get
\begin{equation*}
\widehat{\mathbbm{f}}_{\smallO,s} - \mathbbm{f}_{\smallO,s} \mathbbm{G}_\smallO = \frac{1}{N} (\mathbbm{f}_{\smallO,s} \bbLambda_\smallO' + \mathbbm{e}_{\smallO,s}) \widehat{\bbLambda}_\smallO \mathbbm{D}_\smallO^{-2} - \mathbbm{f}_{\smallO,s} \mathbbm{G}_{\smallO}
= \frac{1}{N} \mathbbm{e}_{\smallO,s} \widehat{\bbLambda}_\smallO\mathbbm{D}_\smallO^{-2}.
\end{equation*}
Taking norms,
\begin{equation*}
\norm{\widehat{\mathbbm{f}}_{\smallO,s} - \mathbbm{f}_{\smallO,s} \mathbbm{G}_{\smallO}} \leq \frac{1}{N} \norm{\mathbbm{e}_{\smallO,s} \widehat{\bbLambda}_\smallO} \normS{\mathbbm{D}_\smallO^{-2}} =  O_p\left(\frac{1}{\lambda^\smallO_{r}}\sqrt{\frac{1}{N}+\frac{1}{T}}\right),
\end{equation*}
where we used $\norm{\mathbbm{e}_{\smallO,s} \widehat{\bbLambda}_\smallO} = O_p(N/\sqrt{T} + \sqrt{N})$ by Lemma~\ref{lemma:Errors}\textit{(iii)} and $\normS{\mathbbm{D}_\smallO^{-2}} = O_p(1/\lambda^\smallO_{r})$ by Lemma~\ref{lemma:MatrixD}\textit{(iii)}.
\end{proof}

\section{Extrapolated Loadings: Proof of Lemma~\ref{lemma:RateLambda}}\label{sec:LemmaRateL}

Before proving Lemma~\ref{lemma:RateLambda}, let us rewrite the matrix $\mathbbm{G}_{\smallO}$. Since
\begin{equation*}
\widehat{\bbLambda}_\smallO = \frac{1}{T_\smallC} \mathbbm{Y}_\smallO'\widehat{\mathbbm{F}}_\smallO = \frac{1}{T_\smallC} (\bbLambda_\smallO \mathbbm{F}_\smallO'+\mathbbm{E}_\smallO')\widehat{\mathbbm{F}}_\smallO,
\end{equation*}
we get
\begin{equation*}
\mathbbm{G}_{\smallO} = \frac{\bbLambda_\smallO'\widehat{\bbLambda}_\smallO}{N} \mathbbm{D}_\smallO^{-2} = \left( \frac{\bbLambda_\smallO'\bbLambda_\smallO}{N} \frac{\mathbbm{F}_\smallO'\widehat{\mathbbm{F}}_\smallO}{T_\smallC} + \frac{\bbLambda_\smallO'\mathbbm{E}_\smallO'\widehat{\mathbbm{F}}_\smallO}{N T_\smallC}\right) \mathbbm{D}_\smallO^{-2}.
\end{equation*}
The above representation of $\mathbbm{G}_{\smallO}$ turns out to be useful in the subsequent pages. Note however that establishing a tight bound on $\mathbbm{G}_{\smallO}$ is not trivial because $\normS{\mathbbm{D}_\smallO^{-2}} = O_p(1/\lambda^\smallO_{r})$ and thus direct application of the spectral norm's submultiplicativity only yields a weak bound.

\begin{lemma}\label{lemma:F}
Under the assumptions of Theorem~1,
\begin{equation*}
\normS{\widehat{\mathbbm{F}}_\smallO-\mathbbm{F}_\smallO\mathbbm{G}_{\smallO}} = O_p\left(
\sqrt{\frac{1}{\lambda^\smallO_{r}}\left(\frac{{r}T}{N} + 1\right)}\right).
\end{equation*}
\end{lemma}

\begin{proof}
Similar to (24) in \citeAppendix{bai2021approximate}, it can be shown that
\begin{equation*}
\widehat{\mathbbm{F}}_\smallO - \mathbbm{F}_\smallO\mathbbm{G}_{\smallO} = \left(\frac{\mathbbm{E}_\smallO \bbLambda_\smallO\mathbbm{F}_\smallO'\widehat{\mathbbm{F}}_\smallO\mathbbm{B}_\smallO^{-1}}{N T_\smallC}+\frac{\mathbbm{E}_\smallO\mathbbm{E}_\smallO'\widehat{\mathbbm{F}}_\smallO\mathbbm{B}_\smallO^{-1}}{N T_\smallC}\right)(\mathbbm{B}_\smallO^{-2}\mathbbm{D}_\smallO^{2})^{-1}\mathbbm{B}_\smallO^{-1}.
\end{equation*}
Consequently,
\begin{align*}
&\normS{\widehat{\mathbbm{F}}_\smallO - \mathbbm{F}_\smallO\mathbbm{G}_{\smallO}} \\
&\leq \left(\frac{\normS{\mathbbm{E}_\smallO \bbLambda_\smallO\mathbbm{B}_\smallO^{-1}}\normS{\mathbbm{B}_\smallO\mathbbm{F}_\smallO'\widehat{\mathbbm{F}}_\smallO\mathbbm{B}_\smallO^{-1}}}{N T_\smallC}+\frac{\normS{\mathbbm{E}_\smallO\mathbbm{E}_\smallO'}\normS{\widehat{\mathbbm{F}}_\smallO}\norm{\mathbbm{B}_\smallO^{-1}}}{N T_\smallC}\right)\\
&\qquad \times \normS{\mathbbm{B}_\smallO^{2}\mathbbm{D}_\smallO^{-2}}\normS{\mathbbm{B}_\smallO^{-1}} \\
&= O_p\left(\sqrt{\frac{{r}T}{\lambda^\smallO_{r} N}}\right) + O_p\left(\frac{N + T}{\lambda^\smallO_{r} N\sqrt{T}}\right) = O_p\left(
\sqrt{\frac{1}{\lambda^\smallO_{r}}\left(\frac{{r}T}{N} + 1\right)}\right),
\end{align*}
since $\normS{\mathbbm{E}_\smallO \bbLambda_\smallO\mathbbm{B}_\smallO^{-1}} = O_p(\sqrt{{r}N T})$ by Lemma~\ref{lemma:Errors}\textit{(i)}, $\normS{\mathbbm{B}_\smallO\mathbbm{F}_\smallO'\widehat{\mathbbm{F}}_\smallO\mathbbm{B}_\smallO^{-1}} = O_p(T)$ by Lemma~\ref{lemma:BFFB}\textit{(i)}, $\normS{\mathbbm{E}_\smallO\mathbbm{E}_\smallO'} = O_p(N + T)$ by Assumption~\ref{ass:Errors}(a), $\normS{\widehat{\mathbbm{F}}_\smallO}^2 = T_\smallC$, $\normS{\mathbbm{B}_\smallO^{-1}} = O(\sqrt{1/\lambda^\smallO_{r}})$, $\normS{\mathbbm{B}_\smallO^{2}\mathbbm{D}_\smallO^{-2}} = O_p(1)$ by Lemma~\ref{lemma:MatrixD}\textit{(ii)} and $1/\lambda^\smallO_{r} = O(\min\{N,T\})$ by Assumption~\ref{ass:Asymptotics}.
\end{proof}

\begin{lemma}\label{lemma:ZF}
Under the assumptions of Theorem~1,
\begin{enumerate}[label=(\roman*)]
\item $\max_{i\leq N}\, \Big\lVert\frac{(\mathbbm{z}_{\smallC,i} + \mathbbm{e}_{\smallC,i})' \mathbbm{F}_\smallO}{T_\smallC}\Big \rVert = O_p\Big(\sqrt{{r}\log(N)/T}\,\Big)$,
\item $\max_{i\leq N} \norm{\mathbbm{e}_{\smallC,i}} = O_p(\sqrt{T})$,
\item $\max_{i\leq N} \norm{\mathbbm{z}_{\smallC,i}} = O_p(\sqrt{T})$,
\end{enumerate}
where $\mathbbm{z}_{\smallC,i} = (Z_t(u_i): t \in \mathcal{T})'$ and $\mathbbm{e}_{\smallC,i} = (e_{ti}^{(0)}: t \in \mathcal{T})'$.
\end{lemma}

\begin{proof} \phantom{}\\
\begin{enumerate}[label=(\roman*)]
\item We rely on similar arguments as in \citeAppendix{fan2011high}. Define $\mathbbm{z}_{ti} = Z_t(u_i)$ and $\mathbbm{f}_{\smallO, tk} = \xi_{tk}^\smallO/\sqrt{{\lambda_k^\smallO}}$. Using Assumption~\ref{ass:Tail}, it can be seen that the choice $\gamma_3 = \gamma_2/3$ satisfies
\begin{equation*}
\max_{k\leq r, i \leq N}\prob(\abs{\mathbbm{z}_{ti}\mathbbm{f}_{\smallO, tk}}> \epsilon) \leq \exp(1-(\epsilon/c)^{\gamma_3}),
\end{equation*}
for some $c>0$ and any $\epsilon > 0$. Moreover, Assumption~\ref{ass:Mixing} and Assumption~\ref{ass:Mixing} imply that $(\mathbbm{z}_{ti}\mathbbm{f}_{\smallO, tk}: t \geq 1)$ is strongly mixing with coefficients $\alpha(h) = \exp(-Ch^{\gamma_1})$. Define $\gamma$ by $1/\gamma = 1/\gamma_1 + 3/\gamma_2$. Since~$\gamma < 1$ by Assumption~\ref{ass:Tail}, a variant of the Bernstein inequality (Theorem~1 in \citeAppendix{merlevede2011bernstein}) gives,
\begin{align*}
&\prob \, \Big(\Big\vert \sum_{t\in \mathcal{T}} \mathbbm{z}_{ti}\mathbbm{f}_{\smallO, tk} \Big\vert > x\Big) \\
& \qquad \leq T_\smallC \exp\left(  -\frac{x^\gamma}{C_1}\right) + \exp\left( -\frac{x^2}{C_2(1+C_3 T_\smallC)}\right)\\
&\qquad  \qquad + \exp \left\lbrace -\frac{x^2}{C_4T_\smallC}\exp\left( \frac{x^{\gamma(1-\gamma)}}{C_5\log(x)^\gamma}\right) \right\rbrace,
\end{align*}
for any $x > 0$. Setting $x = \kappa \sqrt{\log(N)T}$, it follows from Assumption~\ref{ass:Asymptotics} that for $\kappa>0$ large enough,
\begin{equation*}
\prob \, \Big(\Big\vert \sum_{t\in \mathcal{T}} \mathbbm{z}_{ti}\mathbbm{f}_{\smallO, tk} \Big\vert > \kappa \sqrt{\log(N)T}\Big) = o(r^{-1}N^{-1}).
\end{equation*}
The Bonferroni inequality implies
\begin{align*}
&\prob \, \Big(\max_{k \leq {r}, i\leq N} \Big\vert \frac{1}{T_\smallC} \sum_{t\in \mathcal{T}} \mathbbm{z}_{ti}\mathbbm{f}_{\smallO, tk} \Big\vert > \kappa \sqrt{\log(N)/T} \Big)\\
&\qquad \qquad \leq N{r} \max_{k \leq {r}, i\leq N}\, \prob \, \Big(\Big\vert  \sum_{t\in \mathcal{T}} \mathbbm{z}_{ti}\mathbbm{f}_{\smallO, tk} \Big\vert > \kappa \sqrt{\log(N)T}\Big) = o(1).
\end{align*}
Consequently,
\begin{equation*}
\max_{k \leq {r}, i\leq N} \Big\vert \frac{1}{T_\smallC} \sum_{t\in \mathcal{T}} \mathbbm{z}_{ti}\mathbbm{f}_{\smallO, tk} \Big\vert = O_p(\sqrt{\log(N)/T}).
\end{equation*}
A bound for $\max_{k\leq {r},i\leq N} \abs{\frac{1}{T_\smallC} \sum_{t\in \mathcal{T}} e_{ti}^{(0)}\mathbbm{f}_{\smallO,tk}}$ can be derived in a similar way using Assumption~\ref{ass:Tail}. The assertion of the lemma thus follows from
\begin{align*}
\max_{i \leq N}\, \Big\lVert\frac{(\mathbbm{z}_{\smallC,i} + \mathbbm{e}_{\smallC,i})' \mathbbm{F}_\smallO}{T_\smallC}\Big \rVert^2 &=\max_{i \leq N}\, \sum_{k=1}^{r} \Big(\frac{1}{T_\smallC} \sum_{t\in \mathcal{T}} (\mathbbm{z}_{ti} + e_{ti}^{(0)})\mathbbm{f}_{\smallO,tk}\Big)^2\\
&\leq {r} \max_{k \leq {r}, i \leq N} \Big\lvert \frac{1}{T_\smallC} \sum_{t\in \mathcal{T}} (\mathbbm{z}_{ti} + e_{ti}^{(0)})\mathbbm{f}_{\smallO,tk} \Big\rvert^2 \\
&= O_p({r} \log(N) /T).
\end{align*}
\item Using Assumptions~\ref{ass:Errors} and \ref{ass:Tail}(c) with similar arguments as before,
\begin{align*}
\max_{i\leq N} \Big\vert \frac{1}{T_\smallC} \sum_{t\in \mathcal{T}} (\mathbbm{e}_{ti}^2-\expv{[\mathbbm{e}_{ti}^2]}) \Big\vert
=O_p(\sqrt{\log(N)/T}).
\end{align*}
Together with $\max_{i \leq N} \expv{(e_{ti}^{(0)})^2]} = O(1)$ by Assumption~\ref{ass:Errors}(c),
\begin{align*}
\max_{i\leq N} \, \norm{\mathbbm{e}_{\smallC,i}}^2 &= \max_{i\leq N} \, \sum_{t\in \mathcal{T}} \mathbbm{e}_{ti}^2 \\
&\leq \max_{i\leq N} \, \Big\lvert \sum_{t\in \mathcal{T}} (\mathbbm{e}_{ti}^2-\expv{[\mathbbm{e}_{ti}^2]})\Big\rvert + T_\smallC \max_{i\leq N} \expv{[\mathbbm{e}_{ti}^2]} = O_p(T),
\end{align*}
where we have used $\log(N) = O(T)$ by Assumption~\ref{ass:Asymptotics}.
\item Concerning the final assertion, we deduce from stationarity of $Z_t$ and Assumption~\ref{ass:MomentsZ},
\begin{equation*}
\max_{i\leq N} \, \norm{\mathbbm{z}_{\smallC,i}}^2 = \max_{i\leq N} \, \sum_{t\in \mathcal{T}} Z_t(u_i)^2 \leq \sum_{t\in \mathcal{T}} \sup_{u\in [0,1]} \, Z_t(u)^2 = O_p(T).
\end{equation*}
\end{enumerate}
\end{proof}

\begin{proof}[Proof of Lemma \ref{lemma:RateLambda}]
The definition $\widehat{\bblambda}^*_{\smallO,i} = \frac{1}{T_\smallC}\mathbbm{y}_{\smallC,i}' \widehat{\mathbbm{F}}_\smallO$ implies
\begin{equation*}
\widehat{\bblambda}^*_{\smallO,i} =  \bblambda^*_{\smallO,i}\,\frac{ \mathbbm{F}_\smallO'\widehat{\mathbbm{F}}_\smallO}{T_\smallC} + \frac{(\mathbbm{z}_{\smallC,i} + \mathbbm{e}_{\smallC,i})' \widehat{\mathbbm{F}}_\smallO}{T_\smallC}.
\end{equation*}
Rearranging terms, we get
\begin{equation*}
\begin{aligned}
\widehat{\bblambda}^*_{\smallO,i} - \bblambda^*_{\smallO,i} (\mathbbm{H}_{\smallO}')^{-1} = \frac{(\mathbbm{z}_{\smallC,i} + \mathbbm{e}_{\smallC,i})' \mathbbm{F}_\smallO\mathbbm{G}_{\smallO}}{T_\smallC} + \frac{(\mathbbm{z}_{\smallC,i} + \mathbbm{e}_{\smallC,i})' (\widehat{\mathbbm{F}}_\smallO - \mathbbm{F}_\smallO\mathbbm{G}_{\smallO})}{T_\smallC}.
\end{aligned}
\end{equation*}
Taking norms,
\begin{align*}
&\max_{i\leq N} \, \norm{\bblambda^*_{\smallO,i} - \bblambda^*_{\smallO,i} (\mathbbm{H}_{\smallO}')^{-1}} \\
& \qquad \leq \max_{i\leq N}\,\Big\lVert \frac{(\mathbbm{z}_{\smallC,i} + \mathbbm{e}_{c,i})' \mathbbm{F}_\smallO}{T_\smallC}\Big\rVert \normS{\mathbbm{G}_{\smallO}} + \left(\max_{i\leq N}\,\frac{\norm{\mathbbm{z}_{\smallC,i}}}{\sqrt{T_\smallC}} + \max_{i\leq N}\,\frac{\norm{\mathbbm{e}_{\smallC,i}}}{\sqrt{T_\smallC}}\right)\frac{\normS{\widehat{\mathbbm{F}}_\smallO - \mathbbm{F}_\smallO\mathbbm{G}_{\smallO}}}{\sqrt{T_\smallC}}\\
& \qquad = O_p\left(\sqrt{\frac{{r}\log(N)}{\lambda^\smallO_{r} T}}\right) + O_p(1)O_p\left(
\sqrt{\frac{1}{\lambda^\smallO_{r}}\left(\frac{{r}}{N} + \frac{1}{T}\right)}\right)\\
& \qquad =O_p\left(
\sqrt{\frac{r}{\lambda^\smallO_{r}}\left(\frac{1}{N} + \frac{\log(N)}{T}\right)}\right),
\end{align*}
where we have used Lemma~\ref{lemma:F}, Lemma~\ref{lemma:ZF} and Lemma~\ref{lemma:H}\textit{(i)}.
\end{proof}

\section{Consistency of eigenvalues}

Define the diagonal entries of $\mathbbm{D}_\smallO^2$ by $\widehat{\gamma}_1^\smallO \geq \widehat{\gamma}_2^\smallO \geq \dots \geq \widehat{\gamma}_{r}^\smallO$.

\begin{lemma}\label{lemma:Eigenvalues}
Under the Assumptions of Theorem~\ref{thm:Convergence},
\begin{equation*}
\max_{k \leq r}\, \abs{\widehat{\gamma}_k^\smallO - \lambda_k^\smallO} = O_p\left(\sqrt{\frac{1}{N} + \frac{1}{T}}\right).
\end{equation*}
\end{lemma}
\begin{proof}
Recall that $\widehat{\gamma}_k^\smallO$ is the squared $k$-th singular value of the matrix $\frac{\mathbb{Y_\smallO}}{\sqrt{NT_\smallC}}$. As a consequence, it is the $k$-th eigenvalue of $\frac{\mathbb{Y}_\smallO'\mathbb{Y}_\smallO}{NT_\smallC}$ (and $\frac{\mathbb{Y}_\smallO\mathbb{Y}_\smallO'}{NT_\smallC}$). To prove the assertion of the lemma, we compare it to eigenvalues of other operators. To this end, let $\widehat{\gamma}_{X, k}^\smallO$ denote the $k$-th eigenvalue of $\frac{\mathbb{X}_\smallO'\mathbb{X}_\smallO}{NT_\smallC}$ ($\frac{\mathbb{X}_\smallO\mathbb{X}_\smallO'}{NT_\smallC}$). In addition, define $\widehat{\lambda}_k^\smallO$ and $\widehat{\lambda}_k^*$ to be the $k$-th eigenvalues of the operators $\widehat{\Gamma}^\smallO$ and $\Gamma^*$, respectively, which are given by
\begin{equation*}
\widehat{\Gamma}^\smallO f = \frac{1}{T_\smallC} \sum_{t \in \mathcal{T}} \langle \langle \mathcal{X}_t^\smallO, f\rangle \rangle \mathcal{X}_t^\smallO, \qquad \Gamma^* f = \frac{1}{T_\smallC} \sum_{t \in \mathcal{T}} \langle \langle \mathcal{X}_t^*, f\rangle \rangle \mathcal{X}_t^*, \qquad f \in \mathcal{H}.
\end{equation*}
Here, $\mathcal{X}_t^*$ refers to the multivariate random function $$\mathcal{X}_t^*(u) = \sum_{i=1}^{N-1} \mathcal{X}_t^\smallO(u_i)\mathbbm{1} \{u \in [u_i,u_{i+1})\}$$ which results from $\mathcal{X}_t^\smallO$ by discretization. Our aim is to show that
\begin{equation*}
\abs{\widehat{\gamma}_k^\smallO - \lambda_k^\smallO} \leq \abs{\widehat{\gamma}_k^\smallO - \widehat{\gamma}_{X,k}^\smallO} + \abs{\widehat{\gamma}_{X,k}^\smallO - \widehat{\lambda}_k^*} + \abs{\widehat{\lambda}_k^* - \widehat{\lambda}_k^\smallO} +
\abs{\widehat{\lambda}_k^\smallO - \lambda_k^\smallO} = O_p\left(\sqrt{\frac{1}{N} + \frac{1}{T}}\right),
\end{equation*}
uniformly for $k \leq r$.
\begin{itemize}
\item Concerning $\abs{\widehat{\gamma}_k^\smallO - \widehat{\gamma}_{X,k}^\smallO}$, we deduce from Weyl's inequality,
\begin{align*}
\abs{\widehat{\gamma}_k^\smallO - \widehat{\gamma}_{X,k}^\smallO} \leq \frac{\norm{\mathbb{Y}_\smallO\mathbb{Y}_\smallO'-\mathbb{X}_\smallO\mathbb{X}_\smallO'}}{NT_\smallC} &\leq \frac{\norm{\mathbb{E}_\smallO\mathbb{E}_\smallO'}+\norm{\mathbb{X}_\smallO\mathbb{E}_\smallO'} + \norm{\mathbb{E}_\smallO\mathbb{X}_\smallO'}}{NT_\smallC} \\
&= O_p\left(\frac{1}{\sqrt{N}} + \frac{1}{T}\right),
\end{align*}
where we used Assumption~\ref{ass:Errors}(a) and Lemma~\ref{lemma:Errors}\textit{(ii)}.

\item Comparing eigenelements of matrices and discretized operators, it can be shown that $\abs{\widehat{\gamma}_{X,k}^\smallO -\widehat{\lambda}_k^*} = O_p(1/\sqrt{N})$.

\item Concerning $\abs{\widehat{\lambda}_k^* - \widehat{\lambda}_k^\smallO}$, we again infer from Weyl's inequality,
\begin{equation*}
\abs{\widehat{\lambda}_k^* - \widehat{\lambda}_k^\smallO} \leq \norm{\widehat{\Gamma}^* - \widehat{\Gamma}}_{\text{op}}.
\end{equation*}
We then observe that $\norm{\widehat{\Gamma}^* - \widehat{\Gamma}}_{\text{op}}^2$ is bounded by
\begin{equation*}
C \sum_{d, d'=0}^D \int_0^1\int_0^1 \Big(\frac{1}{T_\smallC}\sum_{t \in \mathcal{T}} (X_t^{(d)}(u)X_t^{(d')}(v)-X_t^{*(d)}(u)X_t^{*(d')}(v))\Big)^2 \,\du \, \dv,
\end{equation*}
where we have used the notation $X_t^{(0)}(u) = X_t(u)$, and
\begin{align*}
\int_0^1 \expv[(X_t^{(d)}(u) - X_t^{*(d)}(u))^2] \du &= \sum_{i=1}^{N-1} \int_{u_i}^{u_{i+1}} \expv[(X_t^{(d)}(u) - X_t^{(d)}(u_i))^2] \du \\
&\leq C \sum_{i=1}^{N-1} \int_{u_i}^{u_{i+1}} \abs{u - u_i} \, \du = O\left(\frac{1}{N}\right),
\end{align*}
since $\expv[(X_t^{(d)}(u) - X_t^{(d)}(v))^2] \leq C \abs{u-v}$ by Assumption~\ref{ass:Lipschitz}. Combining the above, we get $\abs{\widehat{\lambda}_k^* - \widehat{\lambda}_k^\smallO} = O_p(1/\sqrt{N})$.

\item Finally, concerning $\abs{\widehat{\lambda}_k^\smallO - \lambda_k^\smallO}$, we deduce from Weyl's inequality and Theorem~\ref{thm:ConsistencyCovariance},
\begin{equation*}
\abs{\widehat{\lambda}_k^\smallO - \lambda_k^\smallO} \leq \norm{\widehat{\Gamma}^\smallO - \Gamma^\smallO}_{\text{op}} = O_p(1/\sqrt{T}).
\end{equation*}
\end{itemize}
\end{proof}

\begin{lemma}\label{lemma:MatrixD}
Under the assumptions of Theorem~\ref{thm:Convergence},
\begin{enumerate}[label=(\roman*)]
\item $\normS{\mathbbm{B}_\smallO^{-2}\mathbbm{D}_\smallO^{2}-\mathbbm{I}_{r}} = o_p(1)$,
\item $\normS{\mathbbm{B}_\smallO^{2}\mathbbm{D}_\smallO^{-2}-\mathbbm{I}_{r}} = o_p(1)$,
\item $\normS{\mathbbm{D}_\smallO^{-2}} = O_p(1/\lambda^\smallO_{r})$.
\end{enumerate}
\end{lemma}

\begin{proof}
\phantom{}\newline
\begin{enumerate}[label=(\roman*)]\item Recall that both $\mathbbm{B}_\smallO$ and $\mathbbm{D}_\smallO$ are diagonal matrices. Then,
\begin{equation*}
\normS{\mathbbm{B}_\smallO^{-2}\mathbbm{D}_\smallO^{2}-\mathbbm{I}_{r}} = \max_{k \leq {r}} \, \Big\vert \frac{\widehat{\gamma}_k^\smallO}{\lambda^\smallO_k}-1 \Big\vert \leq \frac{1}{\lambda^\smallO_{r}} \max_{k \leq {r}}\,  \abs{\widehat{\gamma}_k^\smallO - \lambda^\smallO_k} = O_p\left(\frac{1}{\lambda^\smallO_{r}}\sqrt{\frac{1}{N} + \frac{1}{T}}\right),
\end{equation*}
where the last bound follows from Lemma~\ref{lemma:Eigenvalues}. The assertion then follows from Assumption~\ref{ass:Asymptotics}.
\item Observe
\begin{equation*}
\normS{\mathbbm{B}_\smallO^{2}\mathbbm{D}_\smallO^{-2}-\mathbbm{I}_{r}} \leq \frac{\normS{\mathbbm{B}_\smallO^{-2}\mathbbm{D}_\smallO^{2}-\mathbbm{I}_{r}}}{1-\normS{\mathbbm{B}_\smallO^{-2}\mathbbm{D}_\smallO^{2}-\mathbbm{I}_{r}}} = o_p(1),
\end{equation*}
by the previous result.
\item We get
\begin{align*}
\normS{\mathbbm{D}_\smallO^{-2}} \leq \normS{\mathbbm{B}_\smallO^{-2}} (\normS{\mathbbm{B}_\smallO^{2}\mathbbm{D}_\smallO^{-2} - \mathbbm{I}_{r}} + \normS{\mathbbm{I}_{r}}) &= O(1/\lambda^\smallO_{r})(o_p(1) + O(1)) \\
& = O_p(1/\lambda^\smallO_{r}).
\end{align*}
\end{enumerate}
\end{proof}

\section{Asymptotic equivalence of rotation matrices}\label{sec:Equivalence}

\begin{lemma}\label{lemma:H}
Under the assumptions of Theorem~1,
\begin{enumerate}[label=(\roman*)]
\item $\normS{\mathbbm{G}_{\smallO}} = O_p(\sqrt{1/\lambda^\smallO_{r}})$,
\item $\normS{\mathbbm{H}_{\smallO}} = O_p(\sqrt{1/\lambda^\smallO_{r}})$,
\item $\normS{\mathbbm{H}_{\smallO}^{-1}} = O_p(\sqrt{1/r})$.
\end{enumerate}
\end{lemma}
\begin{proof}
\phantom{}\newline
\begin{enumerate}[label=(\roman*)]
\item Follows from the next lemma and $\normS{\mathbbm{H}_{\smallO}} = O_p(\sqrt{1/\lambda^\smallO_{r}})$.
\item We deduce from Lemma~\ref{lemma:BFFB}\textit{(ii)},
\begin{align*}
\normS{\mathbbm{H}_{\smallO}} &= \Big\lVert\Big(\frac{\widehat{\mathbbm{F}}_\smallO'\mathbbm{F}_\smallO}{T_\smallC}\Big)^{-1}\Big\rVert_2 \leq \normS{\mathbbm{B}_\smallO^{-1}} \Big\lVert\Big( \frac{\mathbbm{B}_\smallO\widehat{\mathbbm{F}}_\smallO'\mathbbm{F}_\smallO \mathbbm{B}_\smallO^{-1} }{T_\smallC}\Big)^{-1}\Big\rVert_2 \normS{\mathbbm{B}_\smallO}
\\
& = O_p\left(\frac{1}{\sqrt{\lambda^\smallO_{r}}}\right)O_p(1)O_p(1) = O_p\left(\frac{1}{\sqrt{\lambda^\smallO_{r}}}\right).
\end{align*}
\item We observe
\begin{equation*}
\normS{\mathbbm{H}_{\smallO}^{-1}} \leq \normS{\widehat{\mathbbm{F}}_\smallO}\normS{\mathbbm{F}_\smallO}/T_\smallC = O_p(\sqrt{{r}}).
\end{equation*}
\end{enumerate}
\end{proof}

\begin{lemma}\label{lemma:Rotation}
Under the assumptions of Theorem~1,
\begin{equation*}
\normS{\mathbbm{G}_{\smallO}-\mathbbm{H}_{\smallO}} = O_p\left(\frac{1}{\lambda^\smallO_{r}}\sqrt{\frac{{r}}{N} + \frac{1}{T}}\right).
\end{equation*}
\end{lemma}

\begin{proof}
It can be shown that
\begin{equation*}
\mathbbm{I}_{r} - \mathbbm{H}_{\smallO}^{-1} \mathbbm{G}_{\smallO} = \Big( \frac{\widehat{\mathbbm{F}}_\smallO'\mathbbm{E}_\smallO\bbLambda_\smallO\mathbbm{F}_\smallO'\widehat{\mathbbm{F}}_\smallO\mathbbm{B}_\smallO^{-2}}{N T_\smallC^2} + \frac{\widehat{\mathbbm{F}}_\smallO'\mathbbm{E}_\smallO\mathbbm{E}_\smallO'\widehat{\mathbbm{F}}_\smallO\mathbbm{B}_\smallO^{-2}}{N T_\smallC^2}\Big) (\mathbbm{B}_\smallO^{-2}\mathbbm{D}_\smallO^2)^{-1}.
\end{equation*}
The first term in the braces is bounded by
\begin{equation*}
\frac{\normS{\widehat{\mathbbm{F}}_\smallO}\normS{\mathbbm{E}_\smallO\bbLambda_\smallO\mathbbm{B}_\smallO^{-1}}\normS{\mathbbm{B}_\smallO\mathbbm{F}_\smallO'\widehat{\mathbbm{F}}_\smallO\mathbbm{B}_\smallO^{-1}}\normS{\mathbbm{B}_\smallO^{-1}}}{N T_\smallC^2} = O_p\left(\sqrt{\frac{{r}}{\lambda^\smallO_{r} N}}\right),
\end{equation*}
as it holds that $\norm{\widehat{\mathbbm{F}}_\smallO}^2_2 = T_\smallC$, $\normS{\mathbbm{E}_\smallO \bbLambda_\smallO\mathbbm{B}_\smallO^{-1}} = O_p(\sqrt{r N T})$ by Lemma~\ref{lemma:Errors}\textit{(i)}, $\normS{\mathbbm{B}_\smallO\mathbbm{F}_\smallO'\widehat{\mathbbm{F}}_\smallO\mathbbm{B}_\smallO^{-1}} = O_p(T)$ by Lemma~\ref{lemma:BFFB}\textit{(i)}, and $\normS{\mathbbm{B}_\smallO^{-1}} = O(\sqrt{1/\lambda^\smallO_{r}})$. The second term in the braces is bounded by
\begin{equation*}
\frac{\normS{\widehat{\mathbbm{F}}_\smallO}\normS{\mathbbm{E}_\smallO\mathbbm{E}_\smallO'}\normS{\widehat{\mathbbm{F}}_\smallO}\normS{\mathbbm{B}_\smallO^{-2}}}{N T_\smallC^2} = O_p\left(\frac{1}{\lambda^\smallO_{r}}\left(\frac{1}{N} + \frac{1}{T}\right)\right),
\end{equation*}
where we have used the additional fact that $\normS{\mathbbm{E}_\smallO\mathbbm{E}_\smallO'} = O_p(N + T)$. Combining the above with $\normS{\mathbbm{B}_\smallO^{2}\mathbbm{D}_\smallO^{-2}} = O_p(1)$ by Lemma~\ref{lemma:MatrixD}\textit{(ii)} and $1/\lambda^\smallO_{r} = O(\min\{N, T\})$ by Assumption~\ref{ass:Asymptotics}, we obtain
\begin{equation*}
\mathbbm{I}_{r} - \mathbbm{H}_{\smallO}^{-1} \mathbbm{G}_{\smallO} = O_p\left(\sqrt{\frac{1}{\lambda^\smallO_{r}}\left(\frac{{r}}{N} + \frac{1}{T}\right)}\right).
\end{equation*}
Finally, the above together with $\normS{\mathbbm{H}_{\smallO}} = O_p(\sqrt{1/\lambda^\smallO_{r}})$ by Lemma~\ref{lemma:H}\textit{(ii)} concludes the proof.
\end{proof}

\section{Supplementary lemmas}

\begin{lemma}\label{lemma:LL}
Under the assumptions of Theorem~1,
\begin{equation*}
\Big\lVert \mathbbm{I}_{r} - \frac{\mathbbm{B}_\smallO^{-1}\bbLambda_\smallO'\bbLambda_\smallO\mathbbm{B}_\smallO^{-1}}{N}\Big\rVert_\text{F} = O\Big(\frac{r}{\lambda_r^\smallO N}\Big).
\end{equation*}
\end{lemma}

\begin{proof}
First, we show that for any $k \leq r$ and $d \in \{0, \dots, D\}$,
\begin{equation}\label{eq:phi_kLipschitz}
\abs{{\varphi_k^\smallO}^{(d)}(u)-{\varphi_k^\smallO}^{(d)}(v)} \leq \frac{C}{\lambda_k^\smallO} \abs{u-v}.
\end{equation}
To this end, note that by the eigenequations,
\begin{equation*}
\Gamma^\smallO({\varphi_k^\smallO})(u) = \lambda_k^\smallO \varphi_k^\smallO(u),
\end{equation*}
whence,
\begin{align*}
{\varphi_k^\smallO}(u)-{\varphi_k^\smallO}(v) &= \frac{1}{\lambda_k^\smallO}\Big(\Gamma^\smallO({\varphi_k^\smallO})(u)-\Gamma^\smallO({\varphi_k^\smallO})(v)\Big) \\
&= \frac{1}{\lambda_k^\smallO}  \langle \langle \expv[ \mathcal{X}^\smallO( \mathcal{X}^\smallO(u)- \mathcal{X}^\smallO(v))], \varphi_k^\smallO \rangle \rangle.
\end{align*}
For the $d$-th component, this reads as
\begin{equation*}
{\varphi_k^\smallO}^{(d)}(u)-{\varphi_k^\smallO}^{(d)}(v) = \frac{1}{\lambda_k^\smallO} \sum_{d'=0}^D \langle \expv[X^{(d')}(X^{(d)}(u)- X^{(d)}(v))], {\varphi_k^\smallO}^{(d')} \rangle_2
\end{equation*}
Using an argument based on Cauchy-Schwarz, it suffices to bound the term $\norm{\expv[X^{(d')}(X^{(d)}(u)- X^{(d)}(v))]}^2$. This, however, follows directly from Assumption~\ref{ass:Lipschitz} and shows \eqref{eq:phi_kLipschitz}.

Accordingly, the eigenfunctions~${\varphi_k^\smallO}^{(d)}$ are Lipschitz continuous with Lipschitz constant $C/\lambda_k^\smallO$. By boundedness, also the product~${\varphi_k^\smallO}^{(d)}{\varphi_\ell^\smallO}^{(d)}$ is Lipschitz with constant $C/\max\{\lambda_k^\smallO, \lambda_\ell^\smallO\}$. A simple property of Riemann sum then implies
\begin{equation*}
\abs{\frac{1}{N}\sum_{d=0}^D \sum_{i=1}^N {\varphi_k^{\smallO}}^{(d)}(u_i){\varphi_\ell^{\smallO}}^{(d)}(u_i)-\delta_{k\ell}} \leq \frac{C}{\lambda_r^\smallO N},
\end{equation*}
uniformly for $k,  \ell \in \{1, \dots, r\}$. We conclude
\begin{align*}
\Big\lVert \mathbbm{I}_{r} - \frac{\mathbbm{B}_\smallO^{-1}\bbLambda_\smallO'\bbLambda_\smallO\mathbbm{B}_\smallO^{-1}}{N}\Big\rVert_\text{F}^2 &= \sum_{k, \ell=1}^{r} \Big(\frac{1}{N}\sum_{d=0}^D \sum_{i=1}^N {\varphi_k^{\smallO}}^{(d)}(u_i){\varphi_\ell^{\smallO}}^{(d)}(u_i)-\delta_{k\ell}\Big)^2 \\
&\leq r^2 \max_{k,\ell \leq r} \Big(\frac{1}{N}\sum_{d=0}^D \sum_{i=1}^N {\varphi_k^{\smallO}}^{(d)}(u_i){\varphi_\ell^{\smallO}}^{(d)}(u_i)-\delta_{k\ell}\Big)^2 \\
&= O\Big(\frac{{r}^2}{(\lambda_r^\smallO)^2N^2}\Big).
\end{align*}
\end{proof}

\begin{lemma}\label{lemma:Loadings}
Under the assumptions of Theorem~\ref{thm:Convergence},
\begin{enumerate}[label=(\roman*)]
\item $\max_{i\leq N} \,  \norm{\bblambda^*_{\smallO,i}} = O(1)$,
\item $\max_{i\leq N} \,  \norm{\widehat{\bblambda}^*_{\smallO,i}} = O_p(1)$.
\end{enumerate}
\end{lemma}

\newpage

\begin{proof}
\phantom{}\newline
\begin{enumerate}[label=(\roman*)]
\item The first assertion follows from 
\begin{equation*}
\max_{i \leq N} \, \norm{\bblambda^*_{\smallO,i}}^2 = \max_{i \leq N} \, \expv\Big[ \mathcal{L}(\mathcal{X}^\smallO)(u_i)^2\Big] \leq 2 \, \expv\Big[ \sup_{u\in [0,1]} X(u)^2 + \sup_{u\in [0,1]} Z(u)^2\Big]
\end{equation*}
which is $O(1)$ due to Assumptions~\ref{ass:MomentsX} and~\ref{ass:MomentsZ}.
\item Observe that by Assumption~\ref{ass:MomentsX} and Lemma~\ref{lemma:ZF}\textit{(ii)},
\begin{equation*}
\max_{i \leq N}\, \norm{\mathbbm{y}_{\smallC,i}}^2 = \max_{i \leq N} \sum_{t \in \mathcal{T}} (X_t(u_i) + e_{ti}^{(0)})^2 = O_p(T).
\end{equation*}
Consequently,
\begin{equation*}
\max_{i\leq N} \,  \norm{\widehat{\bblambda}^*_{\smallO,i}} \leq \max_{i \leq N} \frac{\norm{\mathbbm{y}_{\smallC,i}} \normS{\widehat{\mathbbm{F}}_\smallO}}{T_\smallC} = O_p(1),
\end{equation*}
since $\norm{\widehat{\mathbbm{F}}_\smallO}^2_2 = T_\smallC$.
\end{enumerate}
\end{proof}

\begin{lemma}\label{lemma:Errors}
Under the assumptions of Theorem~\ref{thm:Convergence},
\begin{enumerate}[label=(\roman*)]
\item $\normF{\mathbbm{E}_\smallO\bbLambda_\smallO\mathbbm{B}_\smallO^{-1}}^2 = O_p({r}N T)$,
\item $\normF{\mathbbm{F}_\smallO\bbLambda_\smallO'\mathbbm{E}_\smallO'}^2 = O_p(N T^2)$,
\item $\norm{\mathbbm{e}_{\smallO,t}\widehat{\bbLambda}_\smallO}^2 = O_p(N^2/T + N)$.
\end{enumerate}
\end{lemma}

\begin{proof}
\phantom{}\newline
\begin{enumerate}[label=(\roman*)]
\item Observe that
\begin{align*}
\normF{\mathbbm{E}_\smallO\bbLambda_\smallO\mathbbm{B}_\smallO^{-1}}^2 &= \sum_{t\in \mathcal{T}}\sum_{k=1}^{r} \Big(\sum_{d=0}^{D}\sum_{i=1}^N e^{(d)}_{ti} {\varphi_k^{\smallO}}^{(d)}(u_i)\Big)^2 \\
&\leq (D+1) \sum_{d=0}^{D}\sum_{t\in \mathcal{T}}\sum_{i,j=1}^N e^{(d)}_{ti}e^{(d)}_{tj} \sum_{k=1}^{r} {\varphi_k^{\smallO}}^{(d)}(u_i) {\varphi_k^{\smallO}}^{(d)}(u_j).
\end{align*}
Taking expectation and using Assumption~\ref{ass:Errors}(b) and $\sup_u\abs{{\varphi_k^{\smallO}}^{(d)}(u)} < C$ by Assumption~\ref{ass:Eigenfunctions},
\begin{align*}
\expv[\normF{\mathbbm{E}_\smallO\bbLambda_\smallO\mathbbm{B}_\smallO^{-1}}^2] &\leq  (D+1) T_\smallC \sum_{d=0}^{D}\sum_{i,j=1}^N \abs{\expv[e^{(d)}_{ti}e^{(d)}_{tj}]} \sum_{k=1}^{r} \abs{{\varphi_k^{\smallO}}^{(d)}(u_i) {\varphi_k^{\smallO}}^{(d)}(u_j)} \\
&= O({r}N T).
\end{align*}
\item The factor structure entails $\mathbbm{F}_\smallO\bbLambda_\smallO' = \mathbbm{X}_\smallO$. Therefore, the independence of $X_s^{(d)}(u_i)$ and $e^{(d)}_{ti}$ due to Assumption~\ref{ass:Errors} yields
\begin{align*}
&\expv\left[\norm{\mathbbm{F}_\smallO\bbLambda_\smallO'\mathbbm{e}_{\smallO,t}'}^2\right] \\
&\, = \expv\left[\norm{\mathbbm{X}_\smallO \mathbbm{e}_{\smallO,t}'}^2\right] \leq (D+1)  \sum_{d=0}^{D}\sum_{s  \in \mathcal{T}} \sum_{i,j=1}^N \expv[X_s^{(d)}(u_i)X_s^{(d)}(u_j)]\expv[e^{(d)}_{ti}e^{(d)}_{tj}],
\end{align*}
where we have used the notation $X_s^{(0)}(u) = X_s(u)$. Furthermore, since $\expv[X_s^{(d)}(u_i)X_s^{(d)}(u_j)]$ is bounded, we deduce from Assumption~\ref{ass:Errors}(b),
\begin{equation}\label{eq:Xet}
\expv\left[\norm{\mathbbm{F}_\smallO\bbLambda_\smallO'\mathbbm{e}_{\smallO,t}'}^2\right] \leq C(D+1) T_\smallC\sum_{d=0}^D \sum_{i,j=1}^N \abs{\expv[e_{ti}^{(d)}e_{tj}^{(d)}]} = O_p(N T).
\end{equation}
\item Recall that $\widehat{\bbLambda}_\smallO = \mathbbm{Y}_\smallO'\widehat{\mathbbm{F}}_\smallO/T_\smallC$, whence
\begin{equation}\label{eq:e_tLambda}
\mathbbm{e}_{\smallO,t}\widehat{\bbLambda}_\smallO = \frac{1}{T_\smallC} \left(\mathbbm{e}_{\smallO,t}\mathbbm{X}_\smallO' + \mathbbm{e}_{\smallO,t} \mathbbm{E}_\smallO'\right)\widehat{\mathbbm{F}}_\smallO.
\end{equation}
From~\eqref{eq:Xet}, we get $\norm{\mathbbm{e}_{\smallO,t}\mathbbm{X}_\smallO'}^2 = O_p(N T)$. In addition, 
\begin{align*}
\norm{\mathbbm{e}_{\smallO,t} \mathbbm{E}_\smallO'}^2 &= \sum_{s  \in \mathcal{T}} \Big(\sum_{d=0}^D \sum_{i=1}^N e^{(d)}_{ti}e^{(d)}_{si}\Big)^2 \\
&\leq \Big(\sum_{d=0}^D\sum_{i=1}^N (e^{(d)}_{ti})^2\Big)^2 + \sum_{\substack{s \in \mathcal{T}\\s\neq t}} \Big(\sum_{d=0}^D\sum_{i=1}^N e^{(d)}_{ti}e^{(d)}_{si}\Big)^2.
\end{align*}
Due to Assumption~\ref{ass:Errors}(c), the first term is $O_p(N^2)$. It is relevant only if $t \in \mathcal{T}$ meaning that $X_t$ is completely observed. For the second term, we deduce from Assumption~\ref{ass:Errors}(b),
\begin{equation*}
\expv\left[ \sum_{\substack{s \in \mathcal{T}\\s\neq t}} \Big(\sum_{d=0}^D \sum_{i=1}^N e^{(d)}_{si}e^{(d)}_{ti}\Big)^2\right] \leq (D+1) T_\smallC \sum_{d=0}^D \sum_{i,j=1}^N \abs{\expv[e_{ti}^{(d)}e_{tj}^{(d)}]}^2 = O_p(N T ).
\end{equation*}
Combining the above with~\eqref{eq:e_tLambda} and the fact $\normS{\widehat{\mathbbm{F}}_\smallO}^2 = T_\smallC$, we conclude that $\norm{\mathbbm{e}_{\smallO,t}\widehat{\bbLambda}_\smallO}^2 = O_p(N + N^2/T)$.
\end{enumerate}
\end{proof}

\begin{lemma}\label{lemma:BFFB}
Under the assumptions of Theorem~1,
\begin{enumerate}[label=(\roman*)]
\item $\normS{\mathbbm{B}_\smallO\mathbbm{F}_\smallO'\widehat{\mathbbm{F}}_\smallO\mathbbm{B}_\smallO^{-1}} = O_p(T)$,
\item $\normS{(\mathbbm{B}_\smallO\mathbbm{F}_\smallO'\widehat{\mathbbm{F}}_\smallO\mathbbm{B}_\smallO^{-1})^{-1}} = O_p(1/T)$.
\end{enumerate}
\end{lemma}

\newpage

\begin{proof}
\phantom{}\newline
\begin{enumerate}[label=(\roman*)]
\item Observe that
\begin{equation}\label{eq:lemmaBFFB}
\frac{1}{N T_\smallC^2} \mathbbm{B}_\smallO^{-1} \widehat{\mathbbm{F}}_\smallO'\mathbbm{Y}_\smallO\mathbbm{Y}_\smallO' \widehat{\mathbbm{F}}_\smallO \mathbbm{B}_\smallO^{-1} = \mathbbm{B}_\smallO^{-2} \mathbbm{D}_\smallO^{2}.
\end{equation}
Furthermore, $\normF{\mathbbm{F}_\smallO\bbLambda_\smallO'\mathbbm{E}_\smallO'} = O_p(\sqrt{N}T)$ by Lemma~\ref{lemma:Errors}\textit{(ii)} and $\normS{\mathbbm{E}_\smallO\mathbbm{E}_\smallO'} = O_p(N + T)$ by Assumption~\ref{ass:Errors}(a). Therefore,
\begin{align*}
\mathbbm{Y}_\smallO\mathbbm{Y}_\smallO' &= \mathbbm{F}_\smallO\bbLambda_\smallO'\bbLambda_\smallO\mathbbm{F}_\smallO'+\mathbbm{F}_\smallO\bbLambda_\smallO'\mathbbm{E}_\smallO' + \mathbbm{E}_\smallO\bbLambda_\smallO\mathbbm{F}_\smallO' + \mathbbm{E}_\smallO\mathbbm{E}_\smallO'\\
&= \mathbbm{F}_\smallO\bbLambda_\smallO'\bbLambda_\smallO\mathbbm{F}_\smallO' + O_p(\sqrt{N}T + N).
\end{align*}
Since $\mathbbm{B}_\smallO^{-1} \widehat{\mathbbm{F}}_\smallO'O_p(\sqrt{N}T + N) \widehat{\mathbbm{F}}_\smallO \mathbbm{B}_\smallO^{-1} = O_p(1/\lambda^\smallO_{r}(\sqrt{N}T^2 + N T)) = o_p(N T^2)$ by Assumption~\ref{ass:Asymptotics},
\begin{equation}\label{eq:lemmaBFFB2}
\frac{1}{N T_\smallC^2}\mathbbm{B}_\smallO^{-1} \widehat{\mathbbm{F}}_\smallO'\mathbbm{Y}_\smallO\mathbbm{Y}_\smallO' \widehat{\mathbbm{F}}_\smallO \mathbbm{B}_\smallO^{-1} = \frac{1}{N T_\smallC^2}\mathbbm{B}_\smallO^{-1} \widehat{\mathbbm{F}}_\smallO'\mathbbm{F}_\smallO\bbLambda_\smallO'\bbLambda_\smallO\mathbbm{F}_\smallO' \widehat{\mathbbm{F}}_\smallO \mathbbm{B}_\smallO^{-1} + o_p(1).
\end{equation}
Let us now define the $({r}\times {r})$ matrix $\mathbb{A} = \mathbbm{B}_\smallO^{-1} \widehat{\mathbbm{F}}_\smallO'\mathbbm{F}_\smallO\mathbbm{B}_\smallO/T_\smallC$ to ease notation. A trivial bound of~$\mathbb{A}$ can be obtained by observing $$\normS{\mathbb{A}} \leq \normS{\mathbbm{B}_\smallO^{-1}} \normS{\widehat{\mathbbm{F}}_\smallO}\normS{\mathbbm{F}_\smallO}\normS{\mathbbm{B}_\smallO}/T_\smallC = O_p(\sqrt{r/\lambda^\smallO_{r}}).$$ Then, combining \eqref{eq:lemmaBFFB} with \eqref{eq:lemmaBFFB2},
\begin{equation*}
\mathbb{A}\,\frac{\mathbbm{B}_\smallO^{-1}\bbLambda_\smallO'\bbLambda_\smallO\mathbbm{B}_\smallO^{-1}}{N}\,\mathbb{A}' = \mathbbm{B}_\smallO^{-2} \mathbbm{D}_\smallO^{2} + o_p(1).
\end{equation*}

Consequently,
\begin{align*}
\Big\lVert\mathbb{A}\mathbb{A}'-\mathbbm{I}_{r}\Big\rVert_2 &\leq \normS{\mathbb{A}}^2 \Big\lVert \mathbbm{I}_{r} - \frac{\mathbbm{B}_\smallO^{-1}\bbLambda_\smallO'\bbLambda_\smallO\mathbbm{B}_\smallO^{-1}}{N}\Big\rVert_2 + \normS{\mathbbm{B}_\smallO^{-2} \mathbbm{D}_\smallO^{2}-\mathbbm{I}_{r}} + o_p(1) \\
&= O_p\left(\frac{{r}}{\lambda^\smallO_{r}}\right)O\left(\frac{r}{\lambda_r^\smallO N}\right) + o_p(1) = o_p(1),
\end{align*}
where we have used Lemma~\ref{lemma:LL}, Lemma~\ref{lemma:MatrixD}\textit{(i)}, and Assumption~\ref{ass:Asymptotics}. This finally gives the stronger bound $\normS{\mathbb{A}} = O_p(1)$.
\item It suffices to observe
\begin{equation*}
\normS{\mathbb{A}^{-1}} \leq \frac{\normS{\mathbb{A}}}{1-\normS{\mathbb{A}\mathbb{A}'-\mathbbm{I}_{r}}} = O_p(1).
\end{equation*}
\end{enumerate}
\end{proof}

\section{Strong mixing}\label{sec:Mixing}

We define the mixing coefficients appearing in Assumption~\ref{ass:Mixing}. For a standard reference on mixing, we refer to \citeAppendix{doukhan2012mixing}; the case of random functions is treated in Section 2.4 of \citeAppendix{bosq2000linear}.

\begin{definition}
Let $(\mathcal{X}_t: t \in \mathbb{Z})$ be a time series and define $\mathcal{F}_{-\infty}^k$ and $\mathcal{F}_{k+h}^\infty$ to be the $\sigma$-algebras generated by $(\mathcal{X}_t: t \leq k)$ and $(\mathcal{X}_t: t \geq k+h)$, respectively. Let
\begin{equation*}
\alpha(h) = \sup_{k \in \mathbb{Z}}\, \sup \left\lbrace \abs{\prob(A\cap B) - \prob(A)\prob(B)}: A \in \mathcal{F}_{-\infty}^k, B \in \mathcal{F}_{k+h}^\infty \right\rbrace, \qquad h \in \mathbb{Z},
\end{equation*}
be the mixing coefficients. Then, $(\mathcal{X}_t: t \in \mathbb{Z})$ is said to be $\alpha$-mixing if it holds that $\lim_{h \to \infty}\alpha(h) = 0$.
\end{definition}

\begin{theorem}\label{thm:ConsistencyCovariance}
Let $(\mathcal{X}_t: t \in \mathbb{Z})$ satisfy Assumptions~\ref{ass:Sample}, \ref{ass:Mixing}, and~\ref{ass:MomentsX} and define
\begin{equation*}
\widehat{\Gamma} f = \frac{1}{T_\smallC} \sum_{t \in \mathcal{T}} \langle \langle \mathcal{X}_t, f\rangle \rangle \mathcal{X}_t, \qquad f \in \mathcal{H}.
\end{equation*}
Then,
\begin{equation*}
\expv [ \norm{\widehat{\Gamma}-\Gamma}^2_\textup{op}]  \leq C T_\smallC^{-1}.
\end{equation*}
\end{theorem}

\begin{proof}
One observes that
\begin{align*}
&\norm{\widehat{\Gamma} - \Gamma}_{\text{op}}^2 \\
&\quad \leq C \sum_{d, d'=0}^D \int_0^1\int_0^1 \Big(\frac{1}{T_\smallC}\sum_{t \in \mathcal{T}} (X_t^{(d)}(u)X_t^{(d')}(v)-\expv[X_t^{(d)}(u)X_t^{(d')}(v)])\Big)^2 \,\du \, \dv,
\end{align*}
whence,
\begin{equation*}
\expv[\norm{\widehat{\Gamma} - \Gamma}_{\text{op}}^2] \leq C \sum_{d, d'=0}^D \int_0^1\int_0^1 \var{\Big(\frac{1}{T_\smallC}\sum_{t \in \mathcal{T}} X_t^{(d)}(u)X_t^{(d')}(v)\Big)} \,\du \, \dv.
\end{equation*}
For the proof, it thus suffices to show that uniformly for $u,v \in [0,1]$,
\begin{equation*}
\var{\Big(\frac{1}{T_\smallC}\sum_{t \in \mathcal{T}} X_t^{(d)}(u)X_t^{(d')}(v)\Big)} = O(1/T_\smallC).
\end{equation*}
To this end, we note that the sequence $(X_t^{(d)}(u)X_t^{(d')}(v): t \in \mathcal{T})$ is itself $\alpha$-mixing with coefficients $\alpha(h) = \exp(-Ch^{\gamma_1})$. Thus, Theorem~3 in \citeAppendix{doukhan2012mixing} yields
\begin{align*}
&\text{Cov}\Big(X_t^{(d)}(u)X_t^{(d')}(v), X_s^{(d)}(u)X_s^{(d')}(v)\Big) \\
&\qquad \leq 8 \alpha(\abs{s-t})^{1/\rho} \sup_{u \in [0,1]}\, \expv[X_t^{(d)}(u)^{4+\delta}]^{4/(4+\delta)},
\end{align*}
where $\rho = \frac{4 + \delta}{\delta}$. Since $\sup_u\, \expv[X_t^{(d)}(u)^{4+\delta}]< C$ and the coefficients $\alpha(\abs{s-t})^{1/\rho}$ are summable,
\begin{align*}
&\var{\Big(\frac{1}{T_\smallC}\sum_{t \in \mathcal{T}} X_t^{(d)}(u)X_t^{(d')}(v)\Big)} \\
&\qquad = \frac{1}{T_c^2} \sum_{s,t \in \mathcal{T}} \text{Cov}\Big(X_t^{(d)}(u)X_t^{(d')}(v), X_s^{(d)}(u)X_s^{(d')}(v)\Big) \\
&\qquad \leq \frac{C}{T_c^2} \sum_{s,t \in \mathcal{T}} \alpha(\abs{s-t})^{1/\rho} \leq \frac{C}{T_c} \sum_{h = 0}^\infty \exp\Big(-\frac{Ch^{\gamma_1}}{\rho}\Big) \leq  \frac{C}{T_\smallC}.
\end{align*}
which completes the proof of the theorem.
\end{proof}

\section{Additional simulation results}

In the following, we present the estimated best linear reconstructions of the temperature data set considered in Section~\ref{sec:RealDataApplication}. For more details, we refer to the caption of Figure~\ref{fig:weatherexamples}.

\bibliographyAppendix{ref}
\bibliographystyleAppendix{apalike}

\begin{figure}[h]
\centering
\includegraphics[width=\linewidth]{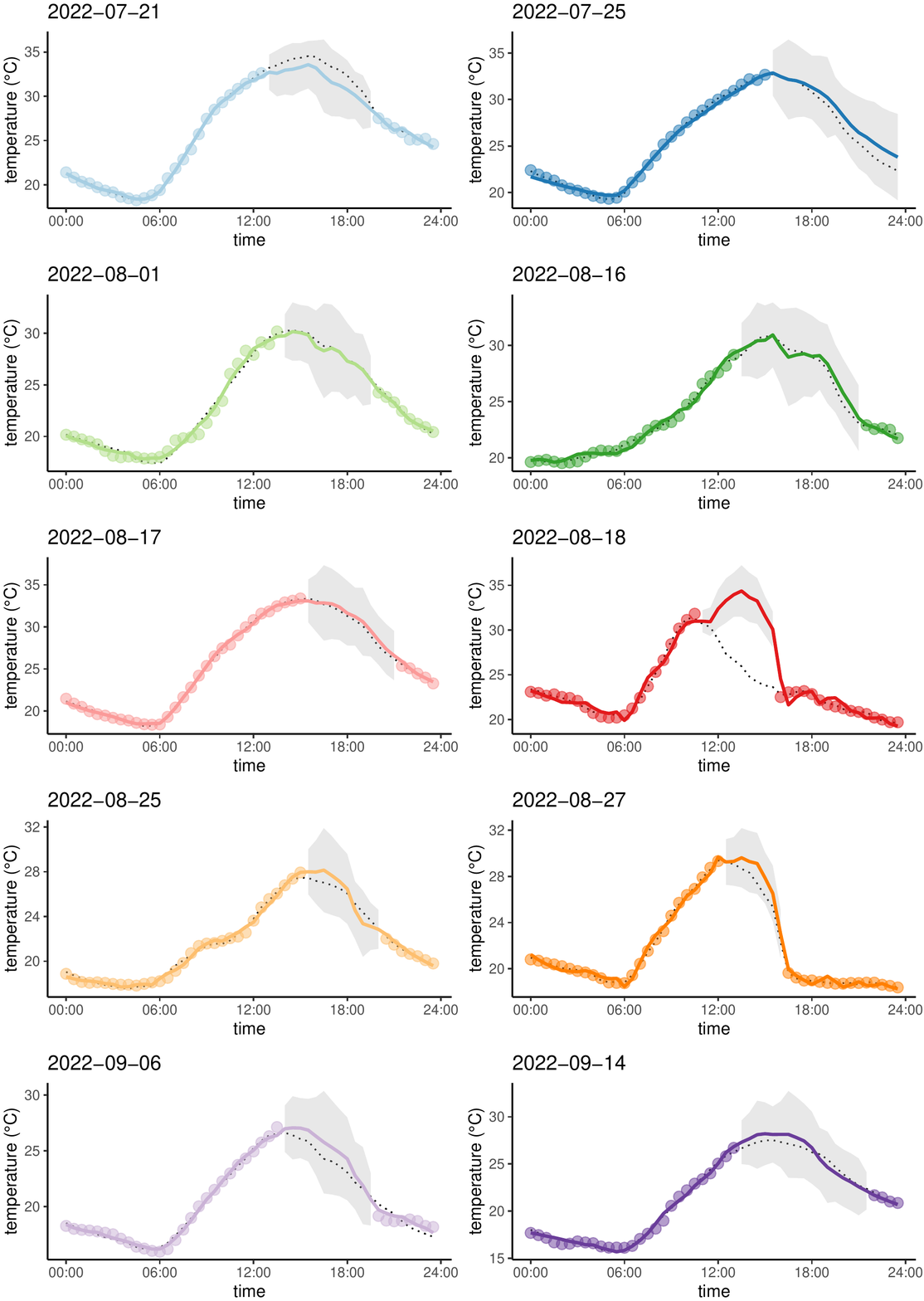}
\caption{Reconstructions additional to Figure~\ref{fig:weatherexamples}.}
\label{fig:weatherreconstructions}
\end{figure}

\newpage

\end{document}